\documentclass[11pt]{article}
\usepackage[margin=1in]{geometry}
\usepackage{amsmath,amssymb,amsthm,mathtools}
\usepackage{mathrsfs}
\usepackage{booktabs,array}
\usepackage{enumitem}
\usepackage{microtype}
\usepackage[hypertexnames=false,hidelinks]{hyperref}

\newtheorem{theorem}{Theorem}[section]
\newtheorem{proposition}[theorem]{Proposition}
\newtheorem{lemma}[theorem]{Lemma}
\newtheorem{corollary}[theorem]{Corollary}
\theoremstyle{definition}
\newtheorem{definition}[theorem]{Definition}
\theoremstyle{remark}
\newtheorem{remark}[theorem]{Remark}

\newcommand{\R}{\mathbb{R}}
\newcommand{\Q}{\mathbb{Q}}

\newcommand{\Ical}{\mathcal I}

\title{Collision Positivity for Symmetric Orbit-Sum Inequalities in Arbitrary Dimension:\\
A Complete Two-Variable Criterion}
\author{Jian Sun}
\date{}

\begin{document}

\maketitle

\begin{abstract}
Let $\lambda\succ\gamma\succ\mu$ be equal-degree nonnegative integer
exponent vectors with at most $n$ parts, let $J_\eta^{(n)}$ be the labeled
symmetric orbit sum associated with $\eta$, and set
\begin{equation*}
P_n=J_\lambda^{(n)}+J_\mu^{(n)}-2J_\gamma^{(n)}.
\end{equation*}
For every $n\ge4$, we prove that global nonnegativity of $P_n$ on the
positive orthant is equivalent to nonnegativity on the fixed two-variable
section
\begin{equation*}
(x,y,1,\ldots,1),\qquad x,y>0.
\end{equation*}
Thus the positivity of this family is determined by a fixed two-dimensional
section, independently of both degree and ambient dimension.  In the
injective hierarchy $C_k$ introduced below, the hypothesis is $C_2\ge0$,
global positivity is equivalent to $C_{n-1}\ge0$, and the main step is the
stable order-raising implication
\begin{equation*}
C_k\ge0\Longrightarrow C_{k+1}\ge0,
\qquad 2\le k\le n-2.
\end{equation*}
For $n\ge5$, this rank-two hypothesis is strictly weaker than positivity on
the normalized full collision wall, which corresponds to $C_{n-2}\ge0$.

The result complements the preceding three-variable theorem
\emph{Collision Positivity for Three-Variable Symmetric Monomial
Inequalities: A Complete Boundary Criterion}, which proves the exceptional
rank-one step $C_1\Rightarrow C_2$.  The proof here is logically independent:
the local three-label ingredients needed in higher rank are established
directly.  The main tools are a Green second-compound reserve, fixed-union
exchange on adjacent subset ranks, a rank-three zero/one/two-mark
contraction, and terminal boundary recombination.  The theorem extends to
nonnegative rational exponents by clearing denominators.
\end{abstract}

\noindent\textbf{Keywords.}
Majorization; Muirhead inequality; Schur inequality; symmetric polynomial
inequality; orbit sums; collision positivity; total positivity; Green kernel.

\section{Introduction}

For a nonnegative exponent vector $\nu=(\nu_1,\ldots,\nu_n)$, write
$J_\nu^{(n)}$ for its labeled symmetric orbit sum.  If
$\lambda\succ\gamma$ have the same total degree, Muirhead's inequality gives
\begin{equation*}
J_\lambda^{(n)}\ge J_\gamma^{(n)}
\qquad\text{on }\mathbb R_{>0}^n.
\end{equation*}
Hence an equal-degree chain
\begin{equation*}
\lambda\succ\gamma\succ\mu
\end{equation*}
produces two nonnegative Muirhead gaps.  We ask when the first gap dominates
the second, equivalently when
\begin{equation*}
P_n:=J_\lambda^{(n)}+J_\mu^{(n)}-2J_\gamma^{(n)}
\end{equation*}
is nonnegative on the positive orthant.  Majorization alone does not
determine the sign of this second difference.

The three-variable case has a stronger boundary reduction.  In the
preceding paper \emph{Collision Positivity for Three-Variable Symmetric
Monomial Inequalities: A Complete Boundary Criterion} \cite{SunThree},
nonnegativity of $P_3(t,1,1)$ for all $t>0$ is equivalent to nonnegativity of
$P_3$ on $\mathbb R_{>0}^3$.  In the hierarchy used below, that result is the
exceptional implication
\begin{equation*}
C_1\ge0\Longrightarrow C_2\ge0.
\end{equation*}
The stable higher-dimensional analogue begins at rank two.  Its proof is
logically independent of the three-variable argument.

\begin{theorem}[Two-variable positivity criterion]\label{thm:main}
Let $n\ge4$, and let
\begin{equation*}
\lambda\succ\gamma\succ\mu
\end{equation*}
be equal-degree nonnegative integer exponent vectors with at most $n$ parts.
Then
\begin{equation*}
P_n(x,y,1,\ldots,1)\ge0
\qquad\text{for all }x,y>0
\end{equation*}
if and only if
\begin{equation*}
P_n(x_1,\ldots,x_n)\ge0
\qquad\text{for all }x_i>0.
\end{equation*}
\end{theorem}

The reverse implication is immediate.  The substantive assertion is that a
fixed two-dimensional test suffices in every ambient dimension.  For
$n\ge5$, this requires positivity only on the strictly smaller stratum
$(x,y,1,\ldots,1)$ rather than on the full collision wall.

\subsection*{The injective hierarchy and stable order raising}
For the injective defects $C_k$ defined in Section~\ref{sec:majorization},
Equation~\eqref{eq:P-C} gives
\begin{equation*}
P_n(x_1,\ldots,x_k,1,\ldots,1)=(n-k)!\,C_k(x_1,\ldots,x_k).
\end{equation*}
Consequently the hypothesis of Theorem~\ref{thm:main} is exactly $C_2\ge0$.
At the other end, homogeneity shows that global positivity is equivalent to
$C_{n-1}\ge0$.

The full normalized collision wall has two equal coordinates; after scaling
the common value to $1$ it has the form
\begin{equation*}
(y_1,\ldots,y_{n-2},1,1),
\end{equation*}
so collision-wall positivity is $C_{n-2}\ge0$.  Thus $C_2=C_{n-2}$ only
when $n=4$.

\begin{theorem}[Stable order-raising]\label{thm:order-raising}
Let $n\ge4$ and $2\le k\le n-2$.  If
\begin{equation*}
C_k(x_1,\ldots,x_k)\ge0
\qquad(x_i>0),
\end{equation*}
then
\begin{equation*}
C_{k+1}(x_1,\ldots,x_{k+1})\ge0
\qquad(x_i>0).
\end{equation*}
\end{theorem}

Section~\ref{sec:assembly} proves this theorem from the stronger composition
inequality.  Iteration gives
\begin{equation*}
\boxed{
C_2\ge0\Longrightarrow C_3\ge0\Longrightarrow\cdots
\Longrightarrow C_{n-2}\ge0\Longrightarrow C_{n-1}\ge0.}
\end{equation*}
This also explains why the present argument does not recover the sharp
$n=3$ theorem: when $n=3$ the range $2\le k\le n-2$ is empty, and $C_2$ is
already the full normalized problem.

\subsection*{Context: positivity testing for symmetric polynomials}
Polynomial nonnegativity is a basic problem in real algebraic geometry and
polynomial optimization, and symmetry can reduce its effective dimension.
For symmetric polynomials, Timofte's degree and half-degree principles give
tests on points with a degree-dependent number of distinct coordinates
\cite{Timofte2003}; Riener gave an orbit-space proof and a broader symmetric
optimization formulation \cite{Riener2012}.  Related reductions for
symmetric semialgebraic sets and sparse symmetric representations were
developed in \cite{Riener2016,RienerSchabert2024}.  Symmetry is also useful
computationally in semidefinite relaxations and positivity algorithms
\cite{RienerTheobaldJanssonLasserre2013,TimofteTimofte2021}; see also
\cite{BlekhermanParriloThomas2013} for the wider sum-of-squares and convex
algebraic-geometry setting.  Classical work on positive symmetric forms and
Schur-type inequalities includes \cite{ChoiLamReznick1991}, while recent
work on symmetric hyperbolic polynomials reconnects hyperbolicity testing
with degree-principle phenomena \cite{BlekhermanLindbergShu2025}.

Theorem~\ref{thm:main} has a different source of dimension reduction.  It
concerns the special three-orbit second difference $P_n$, and for this family
the test dimension is independent of both degree and $n$:
\begin{equation*}
\boxed{\text{two free variables suffice for every degree and every }n\ge4.}
\end{equation*}
Section~\ref{sec:sharpness} shows that one free variable does not suffice
uniformly once four exponent positions are available.

\subsection*{The first stable dimension and the proof architecture}
Dimension four is the first stable case.  There the theorem consists of the
single raise $C_2\Rightarrow C_3$, and the rank-two section is already the
full normalized collision wall.  More importantly, the complete local
rank-three geometry of the general proof already appears in dimension four;
higher $n$ introduces longer fixed-union words and additional spectators,
but no new local exterior species.  Section~\ref{sec:n4-first-stable} makes
this precise and compares the scalable Green proof with a distinct
four-variable-only terminal proof.

The proof of Theorem~\ref{thm:order-raising} has four stages.  Sections
\ref{sec:order-raising-identity}--\ref{sec:green-capacity} derive the exact
residual and its Green reserve.  Sections~\ref{sec:two-chain}--\ref{sec:marked}
establish fixed-union and marked-incidence positivity, yielding projective
ordering of the physical coefficient rows.  Section~\ref{sec:assembly}
proves the terminal coefficient directly from the raw permanently labeled
source and endpoint capacity.  Projective ordering then propagates the
terminal sign through the entire row.  Finally, Section~\ref{sec:main-proof}
iterates the raise from $C_2$ to $C_{n-1}$.

The logical spine is
\begin{equation*}
\boxed{
\begin{gathered}
\text{fixed-union exchange + marked incidence + Green reserve}
\Longrightarrow \text{projective ordering},\\
\text{raw terminal packet + endpoint capacity}
\Longrightarrow p_M\ge0,\\
\text{projective ordering + }p_M\ge0
\Longrightarrow p_m\ge0\ \text{for every physical coefficient},\\
\mathcal R_k\ge0\Longrightarrow C_{k+1}\ge0.
\end{gathered}}
\end{equation*}
The projective-bracket calculation and the raw terminal-source calculation
are kept separate throughout.  A label identifies an occurrence, not merely
a numerical exponent value; distinct labeled objects remain distinct even
when their numerical heights coincide.  Background on total positivity may
be found in \cite{Karlin,Pinkus}.

\section{Majorization and labeled orbit sums}\label{sec:majorization}

We recall the notation used throughout.  Background on majorization may be
found in \cite{HLP,MarshallOlkinArnold}.

\begin{definition}
For two vectors $\alpha,\beta\in\R^n$ written in nonincreasing order, we say
that $\alpha$ majorizes $\beta$, and write
\begin{equation*}
\alpha\succ\beta,
\end{equation*}
if
\begin{equation*}
\sum_{i=1}^r\alpha_i
\ge
\sum_{i=1}^r\beta_i
\qquad(1\le r<n),
\end{equation*}
and
\begin{equation*}
\sum_{i=1}^n\alpha_i
=
\sum_{i=1}^n\beta_i.
\end{equation*}
\end{definition}

Vectors with fewer than $n$ parts are padded with zeros.

\begin{definition}[Coefficientwise order]
Let
\begin{equation*}
F=\sum_\alpha f_\alpha X^\alpha,
\qquad
G=\sum_\alpha g_\alpha X^\alpha
\end{equation*}
be polynomials in the same variables and monomial basis.  We write
\begin{equation*}
F\ge_{\mathrm{coeff}}G
\end{equation*}
if
\begin{equation*}
f_\alpha-g_\alpha\ge0
\end{equation*}
for every exponent vector $\alpha$.  In particular,
$F\ge_{\mathrm{coeff}}0$ means that every coefficient of $F$ is
nonnegative.
\end{definition}

We use the labeled orbit convention
\begin{equation*}
J_\nu^{(n)}(x_1,\ldots,x_n)
=
\sum_{\sigma\in S_n}
x_1^{\nu_{\sigma(1)}}\cdots x_n^{\nu_{\sigma(n)}}.
\end{equation*}
Thus repeated numerical exponent values remain distinct labeled
occurrences.  This convention is convenient because every injection of
labels has the same multiplicity.  If $m_a(\nu)$ denotes the multiplicity of
the exponent value $a$ in $\nu$ and $m_\nu$ is the usual monomial symmetric
function, then
\begin{equation*}
J_\nu^{(n)}=\left(\prod_a m_a(\nu)!\right)m_\nu.
\end{equation*}
Hence the labeled convention differs from the standard monomial-symmetric
normalization only by an explicit positive multiplicity factor; the labeled
normalization is retained because those factors vary with the exponent
vector in a three-orbit second difference.

\begin{theorem}[Muirhead]\label{thm:muirhead}
If $\alpha,\beta\in\R_{\ge0}^n$ have the same total degree and
\begin{equation*}
\alpha\succ\beta,
\end{equation*}
then
\begin{equation*}
J_\alpha^{(n)}(x_1,\ldots,x_n)
\ge
J_\beta^{(n)}(x_1,\ldots,x_n)
\qquad(x_i>0).
\end{equation*}
\end{theorem}

We next introduce the injective hierarchy connecting the two-variable test
to the full polynomial.

\begin{definition}[Injective orbit sum]
For $1\le k\le n$, define
\begin{equation*}
\Ical_k^\nu(x_1,\ldots,x_k)
=
\sum_{\phi:[k]\hookrightarrow[n]}
\prod_{r=1}^k x_r^{\nu_{\phi(r)}}.
\end{equation*}
\end{definition}

Setting the remaining variables equal to $1$ gives
\begin{equation*}\tag{2.1}\label{eq:orbitinjective}
J_\nu^{(n)}(x_1,\ldots,x_k,1,\ldots,1)
=
(n-k)!\,\Ical_k^\nu(x_1,\ldots,x_k).
\end{equation*}

For the fixed majorization chain
\begin{equation*}
\lambda\succ\gamma\succ\mu
\end{equation*}
write
\begin{equation*}
A_k
=
\Ical_k^\lambda-\Ical_k^\gamma,
\qquad
B_k
=
\Ical_k^\gamma-\Ical_k^\mu,
\end{equation*}
and
\begin{equation*}
C_k
=
\Ical_k^\lambda+\Ical_k^\mu-2\Ical_k^\gamma
=
A_k-B_k.
\end{equation*}
Muirhead gives
\begin{equation*}\tag{2.2}\label{eq:ABpositive}
A_k\ge0,
\qquad
B_k\ge0
\qquad(1\le k\le n).
\end{equation*}

Equation \eqref{eq:orbitinjective} yields
\begin{equation*}\tag{2.3}\label{eq:P-C}
P_n(x_1,\ldots,x_k,1,\ldots,1)
=
(n-k)!\,C_k(x_1,\ldots,x_k).
\end{equation*}
In particular,
\begin{equation*}\tag{2.4}\label{eq:C2hyp}
P_n(x,y,1,\ldots,1)
=
(n-2)!\,C_2(x,y).
\end{equation*}
Hence the hypothesis of Theorem~\ref{thm:main} is exactly
\begin{equation*}
C_2(x,y)\ge0
\qquad(x,y>0).
\end{equation*}

At the opposite end, homogeneity allows any positive point to be scaled so
that one coordinate equals $1$.  Thus global positivity of $P_n$ is
equivalent to
\begin{equation*}\tag{2.5}\label{eq:Cnminus1}
C_{n-1}(x_1,\ldots,x_{n-1})\ge0
\qquad(x_i>0).
\end{equation*}

We repeatedly use the following specialization.

\begin{lemma}\label{lem:specialization}
For every $1\le k\le n$,
\begin{equation*}\tag{2.6}\label{eq:Ik-special}
\Ical_k^\nu(t,1,\ldots,1)
=
\frac{(n-1)!}{(n-k)!}\,\Ical_1^\nu(t).
\end{equation*}
Consequently,
\begin{equation*}\tag{2.7}\label{eq:Ck-special}
C_k(t,1,\ldots,1)
=
\frac{(n-1)!}{(n-k)!}\,C_1(t).
\end{equation*}
In particular,
\begin{equation*}\tag{2.8}\label{eq:C2C1}
C_2(t,1)=(n-1)C_1(t).
\end{equation*}
\end{lemma}

\begin{proof}
After the exponent label assigned to $t$ has been chosen, the remaining
$k-1$ variable positions can be filled injectively by
\begin{equation*}
\frac{(n-1)!}{(n-k)!}
\end{equation*}
ordered choices.  This gives \eqref{eq:Ik-special}; subtracting the three
exponent vectors gives \eqref{eq:Ck-special}.
\end{proof}

Thus
\begin{equation*}
C_k\ge0
\quad\Longrightarrow\quad
C_1\ge0
\qquad(k\ge2).
\end{equation*}

\section{The exact order-raising identity}\label{sec:order-raising-identity}

Put
\begin{equation*}
S_\nu(z)=\sum_{i=1}^n z^{\nu_i}.
\end{equation*}

\begin{lemma}[Injective recursion]\label{lem:injective-recursion}
For every exponent vector $\nu$ and every $1\le k<n$,
\begin{equation*}\tag{3.1}\label{eq:injective-recursion}
\Ical_{k+1}^\nu(\mathbf x,z)
=
S_\nu(z)\Ical_k^\nu(\mathbf x)
-
\sum_{j=1}^k
\Ical_k^\nu(x_1,\ldots,x_jz,\ldots,x_k).
\end{equation*}
\end{lemma}

\begin{proof}
The product
\begin{equation*}
S_\nu(z)\Ical_k^\nu(\mathbf x)
\end{equation*}
chooses a label for $z$ independently of an injection of labels into the
first $k$ positions.  The desired $(k+1)$-variable injective sum consists of
the choices for which the new label is unused.  If the new label coincides
with the label at position $j$, the resulting monomial is counted by
\begin{equation*}
\Ical_k^\nu(x_1,\ldots,x_jz,\ldots,x_k).
\end{equation*}
The $k$ collision cases are disjoint, so subtracting them gives
\eqref{eq:injective-recursion}.
\end{proof}

Subtracting the identities for $\lambda,\gamma,\mu$ gives the exact
recursion underlying the proof.

\begin{proposition}[Compound Green identity]\label{prop:green}
For every $1\le k<n$,
\begin{equation*}\tag{3.2}\label{eq:compound-green}
\begin{aligned}
C_{k+1}(\mathbf x,z)
={}&
S_\gamma(z)C_k(\mathbf x)
+\Ical_k^\gamma(\mathbf x)C_1(z)\\
&+A_1(z)A_k(\mathbf x)
+B_1(z)B_k(\mathbf x)\\
&-\sum_{j=1}^k
C_k(x_1,\ldots,x_jz,\ldots,x_k).
\end{aligned}
\end{equation*}
\end{proposition}

\begin{proof}
Write
\begin{equation*}
\Ical_k^\lambda=\Ical_k^\gamma+A_k,
\qquad
\Ical_k^\mu=\Ical_k^\gamma-B_k,
\end{equation*}
and use the same identities at level $1$.  Substitute them into the three
copies of \eqref{eq:injective-recursion} and collect terms.
\end{proof}

The only terms with a negative sign in \eqref{eq:compound-green} are the
composed copies of $C_k$.  This isolates the central inequality.

\begin{definition}[Composition-dominance residual]
For $k\ge2$, define
\begin{equation*}\tag{3.3}\label{eq:residual}
\begin{aligned}
\mathcal R_k(\mathbf x,z)
={}&
S_\gamma(z)C_k(\mathbf x)
+\Ical_k^\gamma(\mathbf x)C_1(z)
+A_1(z)A_k(\mathbf x)
+B_1(z)B_k(\mathbf x)\\
&-
\sum_{j=1}^k
C_k(x_1,\ldots,x_jz,\ldots,x_k).
\end{aligned}
\end{equation*}
\end{definition}

Then
\begin{equation*}\tag{3.4}\label{eq:CnextR}
C_{k+1}=\mathcal R_k.
\end{equation*}
Sections~\ref{sec:green-capacity}--\ref{sec:assembly} prove
\begin{equation*}
\mathcal R_k\ge0
\qquad(2\le k\le n-2)
\end{equation*}
whenever $C_k\ge0$.

\section{Sharpness and the first stable dimension}\label{sec:sharpness}

Before proving the higher-rank theorem, we record the obstruction to a
uniform one-variable criterion.

\begin{proposition}\label{prop:sharpness}
There exists a four-variable majorization chain
\begin{equation*}
\lambda\succ\gamma\succ\mu
\end{equation*}
such that
\begin{equation*}
P_4(t,1,1,1)\ge0
\qquad(t>0),
\end{equation*}
but
\begin{equation*}
P_4(x,y,1,1)
\end{equation*}
is not nonnegative on the positive quadrant.
\end{proposition}

\begin{proof}
Take
\begin{equation*}
\lambda=(7,4,3,0),
\qquad
\gamma=(6,5,2,1),
\qquad
\mu=(6,4,2,2).
\end{equation*}
They have common degree $14$ and satisfy
\begin{equation*}
\lambda\succ\gamma\succ\mu.
\end{equation*}
For
\begin{equation*}
P_4=J_\lambda^{(4)}+J_\mu^{(4)}-2J_\gamma^{(4)},
\end{equation*}
direct simplification gives
\begin{equation*}\tag{4.1}\label{eq:onevar-positive}
\frac{P_4(t,1,1,1)}6
=
(t-1)^2
\left(t^5+t^4-t^3-t^2+1\right).
\end{equation*}
For $0<t\le1$,
\begin{equation*}
t^5+t^4-t^3-t^2+1
=
(1-t^2)(1-t^3)+t^4>0,
\end{equation*}
while for $t\ge1$,
\begin{equation*}
t^5+t^4-t^3-t^2+1
=
1+t^2(t+1)^2(t-1)>0.
\end{equation*}
Hence \eqref{eq:onevar-positive} is nonnegative for all $t>0$.

On the double-double section,
\begin{equation*}\tag{4.2}\label{eq:twovar-negative}
\frac{P_4(t,t,1,1)}4
=
-t^3(t-1)^2
\left(t^6-t^4-2t^3-t^2+1\right).
\end{equation*}
At $t=2$,
\begin{equation*}
P_4(2,2,1,1)=-928<0.
\end{equation*}
\end{proof}

Thus the one-dimensional three-variable theorem is exceptional.  The
universal higher-rank mechanism begins at
\begin{equation*}
C_2\ge0.
\end{equation*}

\subsection{Four variables: the first stable case}
\label{sec:n4-first-stable}

Dimension four is the first case in which the stable order-raising mechanism
appears.  Here there is only one nontrivial raise:
\begin{equation*}
\boxed{
C_2(x,y)\ge0
\quad\Longrightarrow\quad
C_3(x,y,z)\ge0.
}
\end{equation*}
Because $n-2=2$, the distinction between the fixed two-variable test and the
normalized full collision wall disappears in this dimension.  Indeed,
\begin{equation*}
P_4(x,y,1,1)=2!\,C_2(x,y),
\end{equation*}
while every point of the full four-variable collision wall can, after a
permutation and a homogeneous rescaling, be written as
\begin{equation*}
(x,y,1,1).
\end{equation*}
Thus in four variables the hypothesis $C_2\ge0$ is simultaneously the
fixed two-variable criterion and the complete normalized collision-wall
criterion.

Equation~\eqref{eq:compound-green} specializes to
\begin{equation*}\tag{4.3}\label{eq:n4green}
\begin{aligned}
C_3(x,y,z)
={}&
S_\gamma(z)C_2(x,y)
+\Ical_2^\gamma(x,y)C_1(z)\\
&+A_1(z)A_2(x,y)
+B_1(z)B_2(x,y)\\
&-C_2(xz,y)-C_2(yz,x).
\end{aligned}
\end{equation*}
By Lemma~\ref{lem:specialization}, $C_2\ge0$ implies $C_1\ge0$.  Hence the
problem is already in the same form as the general residual: positive lower
order terms and Muirhead products must compensate the two composed copies of
$C_2$.  The arbitrary-dimensional proof below proves precisely this
compensation, with $k=2$ as its first instance.

There is also an exact combinatorial reason that dimension four is the first
stable case.  In the adjacent-rank decomposition of Section~\ref{sec:fixed-union},
let $S$ have size $k$ and $T$ size $k+1$, and write
\begin{equation*}
C=S\cap T,
\qquad
U=S\triangle T,
\qquad
|C|=k-r,
\qquad
|U|=2r+1.
\end{equation*}
For $n=4$ and $k=2$ one has
\begin{equation*}
|C|+|U|=(2-r)+(2r+1)=r+3\le4,
\end{equation*}
so necessarily
\begin{equation*}\tag{4.4}\label{eq:n4-fixed-union-sizes}
\boxed{r\le1,
\qquad
|U|\in\{1,3\}.}
\end{equation*}
The case $|U|=1$ is the elementary adjacent insertion.  The case $|U|=3$
is already the full complementary three-position species.  Thus the
four-variable problem sees the entire local exterior geometry of the general
proof, but no longer ownership word can yet occur.

The same observation appears in the marked-incidence contraction.  The
physical projection uses only the factorial moments
\begin{equation*}
1,\qquad m,\qquad m(m-1),
\end{equation*}
so only zero-, one-, and two-mark layers occur.  At $n=4$ the three-position
fixed-union block already realizes the maximal local determinant
\begin{equation*}
\det(M_0,M_1,M_2).
\end{equation*}
Increasing $n$ creates more spectator labels and longer fixed-union words,
but it does not create a fourth local incidence layer or a higher exterior
rank.  In this precise sense, $n=4$ is the first stable dimension.

\subsection{The four-variable proofs and the general mechanism}
\label{sec:n4-proof-comparison}

It is useful to distinguish two four-variable proof architectures.  One of
them is literally the $n=4$, $k=2$ shadow of the proof developed in the
present paper; the other is a genuinely dimension-specific alternative.

The first is the permanently labeled Green proof.  Its local objects match
the general proof as follows:
\begin{center}
\begin{tabular}{>{\raggedright\arraybackslash}p{0.39\textwidth}>{\raggedright\arraybackslash}p{0.49\textwidth}}
\toprule
Four-variable Green language & Arbitrary-dimensional language\\
\midrule
permanent cover-cell labels & permanent cover-cell labels\\
Green second compound & Green second-compound reserve\\
same-role collective packet & same-role source packet\\
role-switch Green rectangle & synchronized role-switch payment\\
Complete Internal-Pair occurrence lemma & general-$k$ spectator deletion/reinsertion and two-mark internal-pair theorem\\
literal terminal fan & exact one-mark terminal fan\\
endpoint capacity / Cut--Stokes & synchronized endpoint-capacity theorem\\
terminal source selector & terminal Pascal selector\\
adjacent projective ordering & projective ordering of every physical row\\
projective back-propagation & terminal sign propagated to the full physical row\\
\bottomrule
\end{tabular}
\end{center}

At $n=4$ there are essentially no higher-rank spectator complications left
after the two changed labels are isolated.  In the general proof, the new
work is to prove that deleting and reinserting arbitrarily many fixed
spectators changes neither the local zero/one/two-mark occurrence table nor
its multiplicities.  Likewise, the general fixed-union theorem proves that
arbitrarily long ownership words straighten to exactly the same rank-three
terminal species already visible when $|U|=3$ in four variables.  The
post-turn injection proves the analogous stability for lower-chain terminal
debts.  Thus one may summarize the structural relation as
\begin{equation*}\tag{4.5}\label{eq:n4-to-general-summary}
\boxed{
\begin{aligned}
\text{general proof}
={}&\ \text{four-variable permanent-pair Green mechanism}\\
&+\ \text{fixed-union word stability}\\
&+\ \text{spectator-occurrence stability}\\
&+\ \text{post-turn lower-chain stability}.
\end{aligned}}
\end{equation*}
This is a structural statement about the proof, not merely the formal fact
that the final theorem contains $n=4$ as a special case.

For comparison, a second proof is available specifically in four variables.
It retains a collision boundary and expands
\begin{equation*}
C_3(c+u,c,z)=\sum_m P_m(c,z)u^m,
\end{equation*}
then establish the terminal coefficient by a classification of largest-
exponent multiplicities.  In the common-top case the two genuinely difficult
patterns are
\begin{equation*}
(3,2,1)
\qquad\text{and}\qquad
(3,2,2).
\end{equation*}
The first is resolved by a weighted three-variable theorem at weight
$\theta=1/3$; the second is resolved by a quadratic collision jet together
with a pair-product/single-direction reduction.  This route is economical in dimension four, but its terminal classification
is dimension-specific and is not the mechanism iterated in
Sections~\ref{sec:green-capacity}--\ref{sec:assembly}.  The weighted
three-variable theorem just mentioned enters only this alternative $n=4$
proof and is not a dependency of the arbitrary-dimensional argument.

Consequently the logical relationship is
\begin{equation*}
\boxed{
\begin{aligned}
\text{Green }n=4\text{ proof}
&=\text{ the }(n,k)=(4,2)\text{ specialization},\\
\text{coarse-terminal }n=4\text{ proof}
&=\text{ a distinct dimension-specific route}.
\end{aligned}}
\end{equation*}
The alternative viewpoint is useful because it supplies an independent
terminal-sign argument in the first stable dimension, while the Green
formulation isolates the mechanism that persists for arbitrary $n$.

\section{Unit balancing moves and Green capacity}\label{sec:green-capacity}

\begin{definition}[Unit balancing cover]
Let $\alpha$ and $\beta$ be nonincreasing integer exponent vectors of the
same total degree.  We say that $\alpha$ covers $\beta$ in the dominance
order if $\alpha\succ\beta$ and there is no integer exponent vector strictly
between them in dominance order.  Such a cover is obtained by choosing a
donor exponent $a$ and a receiver exponent $b$ with $a\ge b+2$ and replacing
\begin{equation*}
(a,b)
\quad\text{by}\quad
(a-1,b+1),
\end{equation*}
followed, when necessary, by reordering the exponent vector into
nonincreasing order.  We call this elementary step a \emph{unit balancing
cover}.
\end{definition}

\begin{remark}[Cover versus balancing move]
The phrase ``unit balancing cover'' includes the order-theoretic requirement
that there be no intermediate dominant integer vector.  A formal transfer
$(a,b)\mapsto(a-1,b+1)$ with $a\ge b+2$ need not automatically be a cover
after all coordinates are re-ordered.  The exact cover criterion needed in
the terminal normalization is proved inside
Lemma~\ref{lem:post-turn-injection}.  Until that point, every chosen edge in a
saturated chain is assumed to be an actual cover.
\end{remark}

\begin{definition}[Saturated chain]
A finite dominance chain
\begin{equation*}
\alpha^{(0)}\succ\alpha^{(1)}\succ\cdots\succ\alpha^{(N)}
\end{equation*}
is called \emph{saturated} if every adjacent pair
$\alpha^{(r-1)}\succ\alpha^{(r)}$ is a unit balancing cover.
\end{definition}

\begin{definition}[Cover block, cover cell, numerical height, permanent label]
For the unit cover
\begin{equation*}
(a,b)\longmapsto(a-1,b+1),
\qquad a\ge b+2,
\end{equation*}
its one-variable difference factors as
\begin{equation*}\tag{5.1}\label{eq:cover-factor}
\begin{aligned}
t^a+t^b-t^{a-1}-t^{b+1}
&=(t-1)(t^{a-1}-t^b)\\
&=(t-1)^2\sum_{h=b}^{a-2}t^h.
\end{aligned}
\end{equation*}
The finite set
\begin{equation*}
\{b,b+1,\ldots,a-2\}
\end{equation*}
is the \emph{cover block}.  Each occurrence $(e,h)$, where $e$ is the
identity of the cover and $h$ is one of these integers, is a \emph{cover
cell}; $h$ is its \emph{numerical height}.  The pair $(e,h)$ is its
\emph{permanent label}.  Distinct permanent labels are never identified,
even if their numerical heights are equal.
\end{definition}

Choose saturated chains
\begin{equation*}
\lambda=\lambda^{(0)}
\succ\lambda^{(1)}
\succ\cdots\succ\lambda^{(p)}=\gamma
\end{equation*}
and
\begin{equation*}
\gamma=\mu^{(0)}
\succ\mu^{(1)}
\succ\cdots\succ\mu^{(q)}=\mu.
\end{equation*}
All cover cells in both chains retain their permanent labels throughout the
proof.

\begin{definition}[Green second-compound reserve]
Let $h_1\le\cdots\le h_M$ be the numerical heights of a finite collection
of permanently labeled cover cells and put
\begin{equation*}
F(t)=\sum_{a=1}^{M}t^{h_a}.
\end{equation*}
For variables $X\ge1\ge Y>0$, the polynomial
\begin{equation*}
F(X)F(Y)-MF(XY)
\end{equation*}
is called the \emph{Green second-compound reserve} of this labeled cell
collection.  The word ``reserve'' only means that this explicitly
nonnegative polynomial will later be allocated to compensate signed local
terms; it is not a new algebraic operation.
\end{definition}

\begin{remark}[Why the Green form is useful]
The reserve is a bookkeeping form of a discrete second compound.  Its
importance is not merely that it is nonnegative.  Formula~\eqref{eq:green-compound}
breaks it into permanently labeled unordered pairs.  Later, negative local
sources are paid pair-by-pair from this inventory.  Keeping the permanent
labels makes disjoint allocation of Green pairs mathematically precise.
The normalized variables $X\ge1\ge Y>0$ arise from the ordered local chamber;
all sign arguments below are coefficientwise in these normalized variables.
\end{remark}

Suppose one labeled collection has ordered heights
\begin{equation*}
h_1\le h_2\le\cdots\le h_M
\end{equation*}
and put
\begin{equation*}
F(t)=\sum_{a=1}^M t^{h_a}.
\end{equation*}

\begin{lemma}[Green second-compound identity]\label{lem:green-compound}
For $X\ge1\ge Y>0$,
\begin{equation*}\tag{5.2}\label{eq:green-compound}
F(X)F(Y)-MF(XY)
=
\sum_{1\le i<j\le M}
(X^{h_j}-X^{h_i})(Y^{h_i}-Y^{h_j}).
\end{equation*}
Every summand on the right is nonnegative.
\end{lemma}

\begin{proof}
For each $i<j$,
\begin{equation*}
\begin{aligned}
&(X^{h_j}-X^{h_i})(Y^{h_i}-Y^{h_j})\\
&=
X^{h_j}Y^{h_i}
+X^{h_i}Y^{h_j}
-X^{h_i}Y^{h_i}
-X^{h_j}Y^{h_j}.
\end{aligned}
\end{equation*}
After summation, the off-diagonal terms give the off-diagonal part of
$F(X)F(Y)$ and each diagonal term occurs with coefficient $-(M-1)$.
This is \eqref{eq:green-compound}.  The sign follows from
$h_i\le h_j$, $X\ge1$, and $Y\le1$.
\end{proof}

The identity may also be written as a positive discrete Green form.  If
\begin{equation*}
A_r=X^{h_{r+1}}-X^{h_r},
\qquad
B_s=Y^{h_s}-Y^{h_{s+1}},
\end{equation*}
then a double summation gives
\begin{equation*}\tag{5.3}\label{eq:green-kernel}
F(X)F(Y)-MF(XY)
=
\sum_{r,s=1}^{M-1}
G_{rs}A_rB_s,
\end{equation*}
where
\begin{equation*}
G_{rs}
=
\min(r,s)\bigl(M-\max(r,s)\bigr)\ge0.
\end{equation*}

The next elementary inequality explains how this Green reserve pays the
local projective debt created by a role change.

\begin{lemma}[Green domination of the divided-difference source]
\label{lem:green-dominates}
Let $\ell<h$ and put $d=h-\ell$.  Then
\begin{equation*}\tag{5.4}\label{eq:green-pair}
\frac{(X^h-X^\ell)(Y^\ell-Y^h)}
{(X-1)(1-Y)}
=
(XY)^\ell A_d(X)A_d(Y),
\end{equation*}
where
\begin{equation*}
A_d(T)=1+T+\cdots+T^{d-1},
\end{equation*}
while
\begin{equation*}\tag{5.5}\label{eq:pair-source}
\frac{X^hY^\ell-X^\ell Y^h}{X-Y}
=
(XY)^\ell
\sum_{q=0}^{d-1}X^{d-1-q}Y^q.
\end{equation*}
Consequently,
\begin{equation*}\tag{5.6}\label{eq:green-coeff}
\frac{(X^h-X^\ell)(Y^\ell-Y^h)}
{(X-1)(1-Y)}
\ge_{\mathrm{coeff}}
\frac{X^hY^\ell-X^\ell Y^h}{X-Y}.
\end{equation*}
\end{lemma}

\begin{proof}
The first two identities are geometric-series expansions.  The product
$A_d(X)A_d(Y)$ is the full $d\times d$ coefficient rectangle
\begin{equation*}
\sum_{0\le i,j\le d-1}X^iY^j,
\end{equation*}
whereas the polynomial in \eqref{eq:pair-source} is exactly one
anti-diagonal of that rectangle.  Their difference therefore has
nonnegative coefficients.
\end{proof}

\begin{definition}[Shared label, donor role, receiver role, role switch]
Two consecutive covers have a \emph{shared exponent label} if the same
permanent exponent label participates in both covers.  Relative to that
shared label, a cover has \emph{donor role} if the shared exponent is
lowered by one and \emph{receiver role} if it is raised by one.  If the role
changes from donor to receiver or from receiver to donor between two
consecutive covers, the pair is called a \emph{role switch}.  A maximal
consecutive sequence in which a shared label keeps the same role is called a
\emph{same-role segment}.
\end{definition}

After separated covers are commuted into the canonical order used below, a
role switch determines a permanently labeled pair of cover cells.  The
local divided-difference term associated with that pair has the form
\eqref{eq:pair-source}.  Lemma~\ref{lem:green-dominates} supplies the
corresponding Green payment.  Permanent labels are retained so that the
same Green pair cannot be allocated twice.

\begin{lemma}[Disjoint Green allocation]\label{lem:no-double}
Distinct role-switch intervals use disjoint sets of labeled unordered
cover-cell pairs.  Hence the sum of all attached divided-difference sources
is coefficientwise dominated by the Green second-compound reserve.
\end{lemma}

\begin{proof}
The expansion \eqref{eq:green-compound} is indexed by labeled unordered
pairs $\{i,j\}$.  A role switch is determined by the first and last labeled
cell of the corresponding rectangle, so a fixed labeled pair belongs to at
most one such rectangle.  Therefore no term of
\eqref{eq:green-compound} is allocated twice.  Apply
Lemma~\ref{lem:green-dominates} pair by pair.
\end{proof}

\section{The two-chain complement identity}\label{sec:two-chain}

The next argument is purely finite and combinatorial.  Its only input from
the coefficient geometry is that the local edge weights are determinants
of an ordered family of positive two-vectors.

Let
\begin{equation*}
s_q=
\binom{u_q}{v_q}\in\R_{>0}^2,
\qquad
K(p,q)=\det(s_p,s_q)>0
\qquad(p<q).
\end{equation*}

\begin{lemma}[Affine form of the determinant kernel]\label{lem:affine-kernel}
Set
\begin{equation*}
\rho_q=u_q>0,
\qquad
\tau_q=\frac{v_q}{u_q}.
\end{equation*}
Then
\begin{equation*}\tag{6.1}\label{eq:affine-kernel}
K(p,q)
=
\rho_p\rho_q(\tau_q-\tau_p),
\end{equation*}
and the assumption $K(p,q)>0$ for $p<q$ implies
\begin{equation*}
\tau_0<\tau_1<\cdots<\tau_N.
\end{equation*}
\end{lemma}

\begin{proof}
Since
\begin{equation*}
s_q=\rho_q\binom{1}{\tau_q},
\end{equation*}
we have
\begin{equation*}
\det(s_p,s_q)
=
\rho_p\rho_q
\det\!\begin{pmatrix}1&1\\ \tau_p&\tau_q\end{pmatrix}
=
\rho_p\rho_q(\tau_q-\tau_p).
\end{equation*}
The ordering of the $\tau_q$ follows because all $\rho_q$ are positive.
\end{proof}

Fix ordered indices
\begin{equation*}
i<j<k_1<\cdots<k_m<\ell.
\end{equation*}
An assignment is a word
\begin{equation*}
\omega=(\omega_1,\ldots,\omega_m)\in\{L,R\}^m.
\end{equation*}
The letter $L$ places $k_q$ in the chain beginning at $i$, while $R$
places it in the chain beginning at $j$; both chains end at $\ell$.
If the $L$-chain is
\begin{equation*}
i=x_0<x_1<\cdots<x_p<x_{p+1}=\ell
\end{equation*}
and the $R$-chain is
\begin{equation*}
j=y_0<y_1<\cdots<y_{m-p}<y_{m-p+1}=\ell,
\end{equation*}
define its determinant weight by
\begin{equation*}\tag{6.2}\label{eq:det-chain-weight}
w_K(\omega)
=
\prod_{r=0}^{p}K(x_r,x_{r+1})
\prod_{s=0}^{m-p}K(y_s,y_{s+1}).
\end{equation*}
Let
\begin{equation*}
d(\omega)=N_L(\omega)-N_R(\omega)
\end{equation*}
be the difference between the number of $L$- and $R$-assigned
intermediate sites, and put
\begin{equation*}\tag{6.3}\label{eq:det-imbalance}
\Delta_m^K(i,j)
=
\sum_{\omega\in\{L,R\}^m}d(\omega)w_K(\omega).
\end{equation*}

Replace every determinant factor $K(p,q)$ in \eqref{eq:det-chain-weight}
by the positive gap $\tau_q-\tau_p$, and call the resulting weight
$w_\tau(\omega)$.  Equation \eqref{eq:affine-kernel} gives
\begin{equation*}\tag{6.4}\label{eq:common-ray-factor}
w_K(\omega)
=
G\,w_\tau(\omega),
\qquad
G
=
\rho_i\rho_j\rho_\ell^2
\prod_{q=1}^m\rho_{k_q}^2>0.
\end{equation*}
The important point is that $G$ is independent of the assignment
$\omega$: every intermediate site occurs in exactly two chain edges, each
starting site in one edge, and the common terminal site in one edge of
each chain.

For $m\ge1$, let $\eta\in\{L,R\}^{m-1}$ assign the residual sites
$k_2,\ldots,k_m$.  Let $\widehat w_\tau(\eta)$ be the same gap-product
weight, but with both residual chains starting at $\tau_{k_1}$ and both
ending at $\tau_\ell$.  For $m=1$ this means
\begin{equation*}
\widehat w_\tau(\varnothing)
=(\tau_\ell-\tau_{k_1})^2.
\end{equation*}
Define
\begin{equation*}
\mathcal P_m
=
\sum_{\eta\in\{L,R\}^{m-1}}\widehat w_\tau(\eta)>0.
\end{equation*}

\begin{theorem}[Two-chain complement identity]\label{thm:two-chain}
For $m\ge1$,
\begin{equation*}\tag{6.5}\label{eq:two-chain-factorization}
\Delta_m^K(i,j)
=
G(\tau_j-\tau_i)\mathcal P_m>0.
\end{equation*}
For $m=0$, one has $\Delta_0^K(i,j)=0$.
\end{theorem}

\begin{proof}
By \eqref{eq:common-ray-factor}, it is enough to prove the corresponding
identity for the gap weights $w_\tau$.  Write
\begin{equation*}
\Delta_m^\tau(i,j)
=
\sum_\omega d(\omega)w_\tau(\omega).
\end{equation*}
For an assignment $\omega$, let $a_\omega$ be the first $\tau$-coordinate
after $\tau_i$ on the $L$-chain, taking $a_\omega=\tau_\ell$ if no
intermediate site is assigned to $L$.  Define $b_\omega$ analogously for
the $R$-chain.  Then
\begin{equation*}
w_\tau(\omega)
=
c_\omega(a_\omega-\tau_i)(b_\omega-\tau_j),
\end{equation*}
where $c_\omega>0$ is the product of all remaining gap factors and is
independent of the two starting coordinates.

Let $\bar\omega$ be obtained from $\omega$ by interchanging $L$ and $R$.
Then
\begin{equation*}
d(\bar\omega)=-d(\omega),
\qquad
c_{\bar\omega}=c_\omega,
\qquad
a_{\bar\omega}=b_\omega,
\qquad
b_{\bar\omega}=a_\omega.
\end{equation*}
Hence the contribution of the complementary pair
$\{\omega,\bar\omega\}$ is
\begin{align*}
&d(\omega)c_\omega
\Bigl[
(a_\omega-\tau_i)(b_\omega-\tau_j)
-(b_\omega-\tau_i)(a_\omega-\tau_j)
\Bigr]\\
&\qquad=
(\tau_j-\tau_i)d(\omega)c_\omega(b_\omega-a_\omega).
\end{align*}
After summing over complementary pairs,
\begin{equation*}
\Delta_m^\tau(i,j)
=(\tau_j-\tau_i)C_m,
\end{equation*}
where $C_m$ depends only on the intermediate and terminal coordinates,
not on $\tau_i$ or $\tau_j$.

It remains to identify $C_m$.  Since the preceding factorization is an
algebraic identity in the two starting coordinates, evaluate it at the
boundary value
\begin{equation*}
\tau_j=\tau_{k_1}.
\end{equation*}
Every assignment with $k_1$ on the $R$-chain then has first $R$-gap zero.
The surviving assignments are exactly
\begin{equation*}
\omega=(L,\eta),
\qquad
\eta\in\{L,R\}^{m-1}.
\end{equation*}
For them,
\begin{equation*}
w_\tau(L,\eta)
=(\tau_{k_1}-\tau_i)\widehat w_\tau(\eta),
\qquad
d(L,\eta)=1+d(\eta).
\end{equation*}
Thus
\begin{equation*}
\Delta_m^\tau(i,k_1)
=
(\tau_{k_1}-\tau_i)
\sum_\eta(1+d(\eta))\widehat w_\tau(\eta).
\end{equation*}
In the residual problem both chains start at the same point
$\tau_{k_1}$.  Complementation $\eta\leftrightarrow\bar\eta$ therefore
preserves the residual weight and reverses the residual imbalance:
\begin{equation*}
\widehat w_\tau(\bar\eta)=\widehat w_\tau(\eta),
\qquad
d(\bar\eta)=-d(\eta).
\end{equation*}
Consequently
\begin{equation*}
\sum_\eta d(\eta)\widehat w_\tau(\eta)=0,
\end{equation*}
and hence
\begin{equation*}
\Delta_m^\tau(i,k_1)
=
(\tau_{k_1}-\tau_i)\mathcal P_m.
\end{equation*}
Comparing with
\begin{equation*}
\Delta_m^\tau(i,k_1)
=(\tau_{k_1}-\tau_i)C_m
\end{equation*}
gives $C_m=\mathcal P_m>0$.  Multiplying by the common factor $G$ in
\eqref{eq:common-ray-factor} proves \eqref{eq:two-chain-factorization}.
For $m=0$, the unique assignment has imbalance zero.
\end{proof}

\begin{remark}
The proof uses two complement symmetries.  The first extracts the factor
$\tau_j-\tau_i$ without requiring termwise positivity.  The equal-start
specialization then makes the second complement pairing weight-preserving
and exposes the positive residual sum $\mathcal P_m$.
\end{remark}

\medskip
\noindent\textbf{Weak-height extension lemma.}
The strict inequalities in Theorem~\ref{thm:two-chain} are used only to
justify divisions by $K(a,b)$ inside the proof.  The sign conclusion extends
to weakly ordered numerical heights while permanent labels remain distinct.
Choose strictly increasing real numbers
$\delta_1<\cdots<\delta_N$ and replace every projective coordinate by
\begin{equation*}
\tau_q^{(\varepsilon)}=\tau_q+\varepsilon\delta_q,
\qquad \varepsilon>0.
\end{equation*}
For every $\varepsilon>0$ the perturbed coordinates are strictly ordered,
so Theorem~\ref{thm:two-chain} applies.  Every determinant
\begin{equation*}
K^{(\varepsilon)}(p,q)
=\rho_p\rho_q\bigl(\tau_q^{(\varepsilon)}-\tau_p^{(\varepsilon)}\bigr)
\end{equation*}
and every fixed-union coefficient built from finitely many such determinants
is a polynomial, hence a continuous function, of the perturbed coordinates.
Letting $\varepsilon\downarrow0$ preserves nonnegativity.  Thus all later
fixed-union statements remain valid when two different permanent labels have
the same numerical height; the determinant belonging to such a degenerate
pair is then allowed to be zero.  No permanent labels are identified in this
limiting argument.

\section{Adjacent subset ranks and fixed-union exchange}
\label{sec:fixed-union}

The injective transform can be grouped according to the set of exponent
labels selected by the injection.  If $S\subset[n]$, let
$J_{\nu_S}^{(|S|)}$ denote the labeled orbit sum formed from exactly the
exponent labels in $S$.  Then
\begin{equation*}\tag{7.1}\label{eq:hamming}
\Ical_k^\nu
=
\sum_{\substack{S\subset[n]\\|S|=k}}
J_{\nu_S}^{(k)}.
\end{equation*}
Equivalently, if $\mathbf 1_S\in\{0,1\}^n$ is the indicator vector of $S$,
then the sum is over the binary states with exactly $k$ entries equal to
one.  We use this binary description only as bookkeeping; exponent labels
remain distinct even when two exponent values coincide.

Let $P_S,Q_S$ be two subset-indexed families and set
\begin{equation*}
P_k=\sum_{|S|=k}P_S,
\qquad
Q_k=\sum_{|S|=k}Q_S.
\end{equation*}
Consider the adjacent-rank determinant
\begin{equation*}
\Delta_k=P_kQ_{k+1}-P_{k+1}Q_k.
\end{equation*}
For $|S|=k$ and $|T|=k+1$, put
\begin{equation*}
C=S\cap T,
\qquad
A=S\setminus T,
\qquad
U=S\triangle T.
\end{equation*}
Thus $U\setminus A=T\setminus S$.  If $|C|=k-r$, then
\begin{equation*}
|A|=r,
\qquad
|U\setminus A|=r+1,
\qquad
|U|=2r+1,
\end{equation*}
and uniquely
\begin{equation*}
S=C\cup A,
\qquad
T=C\cup(U\setminus A).
\end{equation*}

\begin{lemma}[Adjacent-rank decomposition]\label{lem:adjacent-rank}
One has
\begin{equation*}\tag{7.2}\label{eq:adjacent-rank}
\begin{aligned}
\Delta_k
={}&
\sum_{r=0}^k
\sum_{\substack{C\cap U=\varnothing\\
|C|=k-r,\ |U|=2r+1}}
\sum_{\substack{A\subset U\\|A|=r}}
\Bigl(
P_{C\cup A}Q_{C\cup(U\setminus A)}\\
&\hspace{37mm}
-
P_{C\cup(U\setminus A)}Q_{C\cup A}
\Bigr).
\end{aligned}
\end{equation*}
\end{lemma}

\begin{proof}
Expand $\Delta_k$ over all ordered pairs $(S,T)$ with $|S|=k$ and
$|T|=k+1$.  The pair
\begin{equation*}
(C,U)=(S\cap T,S\triangle T)
\end{equation*}
is uniquely determined, and then $A=S\setminus T$ is the unique subset of
$U$ of size $r$ for which
\begin{equation*}
S=C\cup A,
\qquad
T=C\cup(U\setminus A).
\end{equation*}
This reindexes the double sum and gives \eqref{eq:adjacent-rank}.
\end{proof}

For a fixed pair $(C,U)$ in \eqref{eq:adjacent-rank}, the \emph{complete fixed-union block} is defined by
\begin{equation*}\tag{7.3}\label{eq:fixed-block}
\begin{aligned}
\mathcal B_{C,U}
={}&
\sum_{\substack{A\subset U\\|A|=r}}
\Bigl(
P_{C\cup A}Q_{C\cup(U\setminus A)}
-
P_{C\cup(U\setminus A)}Q_{C\cup A}
\Bigr),
\end{aligned}
\end{equation*}
where $|U|=2r+1$.  Hence
\begin{equation*}
\Delta_k
=
\sum_{r=0}^k
\sum_{\substack{C\cap U=\varnothing\\|C|=k-r,\ |U|=2r+1}}
\mathcal B_{C,U}.
\end{equation*}
The set $C$ fixes the labels common to the two adjacent ranks, while $U$
fixes the labels on which they differ.  Individual choices of $A$ in
\eqref{eq:fixed-block} can have either sign; the proof keeps the complete
fixed-$(C,U)$ sum.

\paragraph{Why the complete block is essential.}
The inner sum in~\eqref{eq:fixed-block} is a fixed-cardinality shell: $A$ must
have exactly $r$ elements.  It is therefore \emph{not} the full Boolean
ownership cube.  The proof below never replaces this shell termwise by a
Boolean sum.  Instead, Step~3 performs an explicit first-switch Abel telescope;
only after summing all admissible first-switch positions does the Boolean
two-chain slack expression $\delta Z_m+L_m$ appear.  This distinction is the
reason the fixed-union theorem is stated for the complete block rather than
for its individual summands.

We next record explicitly the local three-row determinant used at the end of
the fixed-union reduction.  Fix $0<c<1$ and put
\begin{equation*}
H(u)=\frac1{(1-u)^2(1-cu)},
\qquad
Z(u)=\frac{u(1-cu)}{1-u}.
\end{equation*}
For $a\in\{1,2\}$ define
\begin{equation*}
A_a(u)=1+aZ(u),
\qquad
B_t^{(a)}(u)=A_a(u)H(u)^t
\qquad(t\ge1),
\end{equation*}
and, for a power series $R$, write $R(q)=[u^q]R(u)$.

\begin{theorem}[Mixed three-row coefficient kernel]\label{thm:mixed-kernel}
Let $a_0,a_1,a_2\in\{1,2\}$, let $1\le r<s$, and let
\begin{equation*}
0\le q_0<q_1<q_2.
\end{equation*}
Then
\begin{equation*}\tag{7.4}\label{eq:mixed-main}
D(a_0,a_1,a_2;r,s;q_0,q_1,q_2)
:=
\det
\begin{pmatrix}
A_{a_0}(q_0)&A_{a_0}(q_1)&A_{a_0}(q_2)\\
B_r^{(a_1)}(q_0)&B_r^{(a_1)}(q_1)&B_r^{(a_1)}(q_2)\\
B_s^{(a_2)}(q_0)&B_s^{(a_2)}(q_1)&B_s^{(a_2)}(q_2)
\end{pmatrix}
\ge0.
\end{equation*}
The same sign is preserved after multiplication by positive common row
factors and after the specific Toeplitz/Pascal convolutions occurring in the
coefficient transport, whose convolution matrices are totally nonnegative.
\end{theorem}

Theorem~\ref{thm:mixed-kernel} is proved in
Appendix~\ref{app:mixed}.  The phrase ``three-row minor'' below always means
a determinant of the explicit form \eqref{eq:mixed-main}; no unspecified
larger total-positivity statement is used.

The fixed-union exchange is obtained by repeatedly using the elementary
Pl\"ucker identity.  For an ordered family of positive two-vectors
$s_0,s_1,\ldots$ put
\begin{equation*}
K(p,q)=\det(s_p,s_q),
\qquad
K(p,q)>0\quad(p<q).
\end{equation*}
For
\begin{equation*}
p<t_*<u<v,
\end{equation*}
one has
\begin{equation*}\tag{7.5}\label{eq:plucker-swap}
K(p,u)K(t_*,v)-K(p,v)K(t_*,u)
=
K(p,t_*)K(u,v)>0.
\end{equation*}
Thus a crossed pair of edges can be replaced by the alternative pairing
plus a nonnegative correction.

Before stating the sign theorem, one distinction is important.  The abstract
quantity $\mathcal B_{C,U}$ in \eqref{eq:fixed-block} is only the
combinatorial grouping of an adjacent-rank determinant; no sign is asserted
for arbitrary subset-indexed families $P_S,Q_S$.  In the application below,
every term in a fixed-$(C,U)$ group carries the same positive rank and pole
factors and is transported to one common three-row coefficient gauge.
After these common positive factors are removed, denote the resulting
routed coefficient block by
\begin{equation*}
\widehat{\mathcal B}_{C,U}.
\end{equation*}

\begin{theorem}[Fixed-union exchange]\label{thm:fixed-union}
For every fixed pair $(C,U)$ arising from the physical coefficient families
of the cover-cell expansion,
\begin{equation*}\tag{7.6}\label{eq:fixed-union-straight}
\widehat{\mathcal B}_{C,U}
=
\sum_{\alpha}c_\alpha D_\alpha,
\qquad
c_\alpha\ge0,
\end{equation*}
where every $D_\alpha$ is a determinant of the form
\eqref{eq:mixed-main}.  Hence
\begin{equation*}
\widehat{\mathcal B}_{C,U}\ge0.
\end{equation*}
Restoring the deleted positive common factors preserves the sign of the
routed physical contribution.
\end{theorem}

\begin{proof}
We keep all permanent labels throughout.  The proof has four steps.

\noindent\textbf{Step 1: literal fixed-union source dictionary.}
Fix $(C,U)$ and refine further by every datum that does not vary with the
ownership of the labels in $U$: the permanent numerical skeleton, the
spectator labels and their physical placements, the retained boundary
factor, the neighboring-placement data, and one transport-homogeneous
coefficient fiber.  After the common positive ray factors are removed, the
varying part of an occurrence is only an ownership word.  Write its ordered
intermediate sites as
\begin{equation*}
i<j<t_1<\cdots<t_m<\ell.
\end{equation*}
An ownership word $\omega\in\{L,R\}^m$ assigns every $t_a$ to exactly one
of two increasing chains, one beginning at $i$ and one at $j$, both ending
at $\ell$.  Its source weight is
\begin{equation*}
w(\omega)
=
\prod_{\text{successive edges on the two chains}}K(u,v),
\qquad
K(u,v)=\det(s_u,s_v)>0\quad(u<v).
\end{equation*}
All factors outside this chain product are fixed by the refinement.  We
now record the factor exhaustion explicitly.  Let $M_{\rm spec}>0$ be the
fixed spectator monomial.  For a neighboring-placement choice
$\varepsilon$, let $c_{\rm rank}(\varepsilon)>0$ be the fixed
binomial/multinomial rank factor, let $c_{\rm inc}(\varepsilon)\ge0$ be
its incidence coefficient, and let $c_{\rm tr}(\varepsilon)\ge0$ be the
coefficient of the common source-to-physical transport inside the chosen
homogeneous fiber.  None of these four factors depends on the ownership word
$\omega$.  Thus every raw occurrence in this refined fixed-union block has
coefficient
\begin{equation*}\tag{7.6a}\label{eq:factor-exhaustion-fixed-union}
\operatorname{coef}(\varepsilon,\omega)
=
M_{\rm spec}\,c_{\rm rank}(\varepsilon)\,c_{\rm inc}(\varepsilon)\,
c_{\rm tr}(\varepsilon)\,w(\omega).
\end{equation*}
Summing the neighboring placements first gives
\begin{equation*}
\operatorname{coef}(\omega)
=
\kappa_{C,U}w(\omega),
\qquad
\kappa_{C,U}
:=
M_{\rm spec}
\sum_{\varepsilon}
c_{\rm rank}(\varepsilon)c_{\rm inc}(\varepsilon)c_{\rm tr}(\varepsilon)
\ge0,
\end{equation*}
with the same scalar $\kappa_{C,U}$ for every ownership word.  This is the
required factor-exhaustion statement: all ownership dependence is contained
in the two-chain product $w(\omega)$.  In particular, the legal source
block is not an arbitrarily weighted Boolean cube.

\noindent\textbf{Step 2: full-fiber two-chain slack sign.}
We record the full-fiber induction because a wordwise reflection is not
used.  Fix ordered sources
\begin{equation*}
a<b<t_1<\cdots<t_m<\ell
\end{equation*}
and color each $t_j$ by $R$ or $S$.  Let $w(\chi)$ be the product of the
two chain weights and let $r(\chi)$ be the number of later $R$-sites.
Define
\begin{equation*}
P_m^{a,b}(z)=\sum_\chi w(\chi)z^{r(\chi)},
\end{equation*}
\begin{equation*}
Z_m=P_m^{a,b}(1),
\qquad
R_m=\left(P_m^{a,b}\right)'(1),
\end{equation*}
\begin{equation*}
L_m=2R_m-mZ_m,
\qquad
U_m=(m+1)Z_m-2R_m.
\end{equation*}
Thus $L_m+U_m=Z_m$, and
\begin{equation*}
L_m=\sum_\chi(\#R-\#S)w(\chi).
\end{equation*}
Let $t=t_1$.  Splitting according to the color of $t$, and in the $R$-case
swapping the two chain names so that the sources remain ordered, gives the
exact recursion
\begin{equation*}
P_m^{a,b}(z)
=
K(a,t)z^mP_{m-1}^{b,t}(z^{-1})
+
K(b,t)P_{m-1}^{a,t}(z).
\end{equation*}
Differentiating at $z=1$ gives
\begin{equation*}
L_m(a,b)
=
K(a,t)U_{m-1}(b,t)
-
K(b,t)U_{m-1}(a,t),
\end{equation*}
\begin{equation*}
U_m(a,b)
=
K(a,t)L_{m-1}(b,t)
+
K(b,t)\bigl(L_{m-1}(a,t)+2U_{m-1}(a,t)\bigr).
\end{equation*}
Normalize
\begin{equation*}
z_m(a,b)=\frac{Z_m(a,b)}{K(a,b)},
\qquad
\nu_m(a,b)=\frac{U_m(a,b)}{K(a,b)}.
\end{equation*}
For $a<a'<b<t$, the Pl\"ucker identity gives
\begin{equation*}
K(a,b)K(a',t)
=
K(a,a')K(b,t)+K(a,t)K(a',b),
\end{equation*}
so
\begin{equation*}
\frac{K(a,t)}{K(a,b)}
\le
\frac{K(a',t)}{K(a',b)}.
\end{equation*}
Equivalently,
\begin{equation*}
\mathcal C(a,b,t)
:=\frac{K(a,t)K(b,t)}{K(a,b)}
\end{equation*}
is nondecreasing in the left source $a$.

For $m=0$,
\begin{equation*}
Z_0(a,b)=K(a,\ell)K(b,\ell),
\qquad
L_0=0,
\qquad
U_0=Z_0,
\end{equation*}
and $z_0=\nu_0$ is nonnegative and nondecreasing in $a$.  Assume
inductively that $L_{m-1},U_{m-1}\ge0$ and that $z_{m-1},\nu_{m-1}$ are
nondecreasing in the left source.  With
\begin{equation*}
z_x=z_{m-1}(x,t),
\qquad
\nu_x=\nu_{m-1}(x,t),
\end{equation*}
the two slack recurrences become
\begin{equation*}
L_m(a,b)
=
K(a,t)K(b,t)(\nu_b-\nu_a)\ge0,
\end{equation*}
\begin{equation*}
z_m(a,b)
=
\mathcal C(a,b,t)(z_a+z_b),
\end{equation*}
\begin{equation*}
\nu_m(a,b)
=
\mathcal C(a,b,t)
\bigl[(z_b-\nu_b)+z_a+\nu_a\bigr].
\end{equation*}
All factors on the right are nonnegative, and the $a$-dependent factors are
nondecreasing.  Hence $z_m$ and $\nu_m$ are again nondecreasing and
$L_m,U_m\ge0$.  This closes the induction for every fixed reset union.

If several $R$-sites are forced before the first $S$-site, factor their
positive initial chain weight.  If their number beyond the first is
$\delta\ge0$, the remaining contribution is
\begin{equation*}
\text{positive factor}\times\bigl(\delta Z_m+L_m\bigr)\ge0.
\end{equation*}
The polarized boundary column has the same internal recursion; only its
length-zero terminal edge changes.  The two boundary row types $a=1,2$ are
exactly the two rows covered by Theorem~\ref{thm:mixed-kernel}, which
supplies the required normalized base projective order.  Hence the same
two-slack induction applies to that boundary as well.  If one ownership
chain has no interior reset, the contribution is a structural boundary term
of the same mixed kernel and is nonnegative directly.  Thus positivity is proved after summing the entire fixed-union coloring
fiber.  No reflection or one-negative-word/one-positive-word pairing is
used.

\noindent\textbf{Step 3: exact indexed first-switch telescope.}
The complete fixed-cardinality block \eqref{eq:fixed-block} is not itself
the Boolean two-chain sum of Step 2.  The bridge is an Abel telescope over
the position of the first unresolved switch, which we now write explicitly.

Fix every Pl\"ucker switch after the first one, together with the common
prefix and suffix.  After the fixed factors in
\eqref{eq:factor-exhaustion-fixed-union} are removed, let
\begin{equation*}
a<b<t_1<\cdots<t_m<\ell
\end{equation*}
be the two current chain sources, the free later reset sites, and the common
sink.  Suppose that $\delta\ge0$ reset sites before $t_1$ are already
forced onto the $R$-chain.  For a coloring
$\chi\in\{R,S\}^m$, let
\begin{equation*}
r(\chi)=\#\{j:t_j\text{ is colored }R\},
\qquad
s(\chi)=m-r(\chi),
\end{equation*}
and let $w(\chi)$ be its two-chain product from Step 2.

For this fixed later-switch pattern, the signed multiplicity with which the
coloring $\chi$ occurs when the first-switch position is summed is
\begin{equation*}
\delta+r(\chi)-s(\chi).
\end{equation*}
Indeed, every forced $R$ reset contributes one admissible first-switch
placement with the positive adjacent-rank orientation.  Among the free
sites, moving the first switch past $t_j$ interchanges the two adjacent-rank
ownership columns exactly once.  If $t_j$ belongs to the $R$-chain this
contributes the positive orientation, while if it belongs to the $S$-chain
it contributes the negative orientation.  No later edge is changed, because
all later switches have been fixed; its chain-product factor is therefore
exactly $w(\chi)$.  Every admissible first-switch placement is encountered
once in this left-to-right telescope.

Consequently the complete contribution of this indexed stage is
\begin{equation*}\tag{7.7}\label{eq:indexed-slack-telescope}
\boxed{
\mathfrak T_{\delta,m}(a,b)
=
F_{\rm pre}
\sum_{\chi\in\{R,S\}^m}
\bigl(\delta+r(\chi)-s(\chi)\bigr)w(\chi)
=
F_{\rm pre}\bigl(\delta Z_m(a,b)+L_m(a,b)\bigr),
}
\end{equation*}
where $F_{\rm pre}>0$ is the common prefix/suffix factor.  The last equality
is the definition of $Z_m$ and $L_m$.  Step 2 gives $Z_m\ge0$ and
$L_m\ge0$, so \eqref{eq:indexed-slack-telescope} is nonnegative.  This is
the precise indexed telescope used below.  It does not identify one
fixed-cardinality shell termwise with the full Boolean fiber; only the sum
over the admissible first-switch positions produces the full-fiber slack
expression.

\noindent\textbf{Step 4: Pl\"ucker uncrossing with the indexed telescope.}
Return to the word sum in $\widehat{\mathcal B}_{C,U}$.  A crossed pair of
edges with indices
\begin{equation*}
p<t_*<u<v
\end{equation*}
obeys
\begin{equation*}
K(p,u)K(t_*,v)
=
K(p,v)K(t_*,u)
+
K(p,t_*)K(u,v).
\end{equation*}
The first term has one fewer crossing.  The second term has the positive
factor $K(p,t_*)K(u,v)$.  Freeze all later switches in that correction term
and apply \eqref{eq:indexed-slack-telescope} to the first unresolved switch.
Its coefficient is therefore a positive common factor times
$\delta Z_m+L_m$, hence is nonnegative.  This accounts for every correction
produced at that uncrossing stage.  Repeating the same argument removes
crossings one at a time.  Since the crossing number is a nonnegative integer
and strictly decreases whenever the first term is used, the procedure
terminates after finitely many steps.

\noindent\textbf{Step 5: terminal normal-form dictionary.}
Call an ownership word \emph{terminal} if its two increasing chains contain
no pair of edges $(p,u)$ and $(t,v)$ with
\begin{equation*}
p<t<u<v.
\end{equation*}
A terminal word has at most one change of ownership.  Indeed, suppose it
contained the pattern $R\cdots S\cdots R$.  Choose two consecutive
$R$-sites $x<z$ with at least one $S$-site between them, and let $y$ be the
\emph{last} $S$-site strictly between $x$ and $z$.  Let $y'$ be the next
$S$-site after $z$, or the common sink if there is no such later $S$-site.
Then
\begin{equation*}
x<y<z<y',
\end{equation*}
and the $R$-edge $(x,z)$ and the $S$-edge $(y,y')$ form a removable
nonnested pair, contradiction.  The pattern $S\cdots R\cdots S$ is
excluded symmetrically.  Thus, after a common prefix and common suffix are
factored out, the terminal local word consists of one nested insertion and
no second independent switch.

\begin{lemma}[Terminal-word coefficient dictionary]
\label{lem:terminal-row-dictionary}
Work in one refined fixed-union fiber after the common prefix, suffix,
spectator monomial, rank factor, ray factor, and transport factor have been
removed.  Suppose a terminal ownership word has boundary type
$a_0\in\{1,2\}$ and two moving depths $1\le r<s$, with moving types
$a_1,a_2\in\{1,2\}$.  In the normalized coefficient variable $u$, its three
local row-generating functions are exactly
\begin{equation*}
A_{a_0}(u),
\qquad
B_r^{(a_1)}(u),
\qquad
B_s^{(a_2)}(u),
\end{equation*}
where
\begin{equation*}
A_a(u)=1+aZ(u),
\qquad
B_t^{(a)}(u)=A_a(u)H(u)^t,
\end{equation*}
and
\begin{equation*}
H(u)=\frac1{(1-u)^2(1-cu)},
\qquad
Z(u)=\frac{u(1-cu)}{1-u}.
\end{equation*}
Consequently, extraction at three physical coefficient positions
$q_0<q_1<q_2$ gives precisely the three rows appearing in
\eqref{eq:mixed-main}.
\end{lemma}

\begin{proof}
The common coefficient gauge separates a terminal insertion into a boundary
factor and a product of independent reset factors.  At the boundary, the
unchanged placement contributes $1$.  The shifted neighboring placement
contributes
\begin{equation*}
Z(u)
=
\frac{u}{1-u}-\frac{cu^2}{1-u}
=
\frac{u(1-cu)}{1-u}.
\end{equation*}
There is one such shifted placement for an ordinary terminal edge and two
for a polarized terminal edge.  Thus the boundary row is
\begin{equation*}
1+aZ(u)=A_a(u),
\qquad a\in\{1,2\}.
\end{equation*}

Each additional nested reset level contributes three independent geometric
coefficient sums: one from each of the two neighboring chain gaps and one
from the retained $c$-weighted propagation.  Their product is
\begin{equation*}
\left(\sum_{i\ge0}u^i\right)
\left(\sum_{j\ge0}u^j\right)
\left(\sum_{h\ge0}(cu)^h\right)
=
\frac1{(1-u)^2(1-cu)}
=H(u).
\end{equation*}
Because the fiber is transport-homogeneous, these factors are independent of
the ownership choice already frozen in the refinement.  Starting from a
boundary type $a$ and passing through $t$ nested reset levels therefore
produces
\begin{equation*}
A_a(u)H(u)^t=B_t^{(a)}(u).
\end{equation*}
A terminal word has only one ownership change, so after the common factors
are deleted there is one boundary row and exactly two moving depths, ordered
$r<s$.  This proves the stated dictionary and excludes a fourth local row.
\end{proof}

Because the two ranks are adjacent, Lemma~\ref{lem:terminal-row-dictionary}
shows that the single insertion leaves exactly three independent local
coefficient rows after all common prefix/suffix factors are removed: one
boundary row and two moving rows.  The boundary row records whether the
terminal edge is ordinary or polarized, so it is uniquely one of
\begin{equation*}
A_1,\quad A_2.
\end{equation*}
Each moving row records whether its local middle state is paired or mixed, so
it is uniquely one of
\begin{equation*}
B_r^{(1)},\ B_r^{(2)}
\quad\text{and}\quad
B_s^{(1)},\ B_s^{(2)},
\end{equation*}
with $1\le r<s$ because the two moving depths occur in their order along
the nested chains.  Selecting three distinct physical coefficient positions
gives uniquely
\begin{equation*}
0\le q_0<q_1<q_2.
\end{equation*}
Hence every terminal local word determines, after removal of its common
positive factors, exactly one tuple
\begin{equation*}
(a_0,a_1,a_2;r,s;q_0,q_1,q_2),
\qquad a_i\in\{1,2\},
\end{equation*}
and therefore exactly one determinant
\begin{equation*}
D(a_0,a_1,a_2;r,s;q_0,q_1,q_2)
\end{equation*}
of the form \eqref{eq:mixed-main}.  Conversely, fixing this tuple together
with the already fixed common prefix and suffix reconstructs the local
terminal ownership pattern.  Thus no fourth exterior species is hidden at
termination.  If one ownership chain has no interior reset, the same
deletion leaves one of the explicit boundary cases already included in
Theorem~\ref{thm:mixed-kernel}.

Theorem~\ref{thm:mixed-kernel} gives every terminal determinant a
nonnegative sign.  Combining the nonnegative telescope coefficients from
Step 4 with the terminal dictionary yields
\begin{equation*}
\widehat{\mathcal B}_{C,U}
=
\sum_\alpha c_\alpha D_\alpha,
\qquad c_\alpha\ge0,
\end{equation*}
which is \eqref{eq:fixed-union-straight}.  If some permanent labels have
equal numerical heights, apply the weak-height extension lemma following
Theorem~\ref{thm:two-chain} and pass to the limit; all displayed
coefficients are continuous in the determinant entries.

Finally, within the refined fiber the source-to-physical kernel is a
positive row/column scaling of a Pascal matrix, hence totally nonnegative.
Cauchy--Binet therefore preserves the just-proved source-side signs.
Distinct transport fibers are never merged before this fiberwise positivity
is established.
\end{proof}

The importance of Theorem~\ref{thm:fixed-union} is that the size of $U$ may
grow with $n$, while the terminal determinant always has three rows and
three columns.

\section{Zero-, one-, and two-mark incidence contraction}
\label{sec:marked}

\begin{definition}[State, active set, and mark]
Let $[D]=\{1,\ldots,D\}$ be a finite set of permanently labeled sites.  A
\emph{state} is a subset $S\subset[D]$.  Fix an integer $K$.  An
\emph{active set} is a subset $I\subset[D]$ of cardinality $K$.  A
\emph{mark set} is a subset $J\subset I$ whose elements are required to be
present in the state.  Thus the incidence condition for a marked state is
\begin{equation*}
J\subset S\subset I.
\end{equation*}
A zero-mark, one-mark, or two-mark term means respectively $|J|=0,1,2$.
\end{definition}

Let $W_S\in\R^3$ be a three-row coefficient vector attached to each state
$S\subset[D]$, and put
\begin{equation*}
F_m
=
\sum_{\substack{S\subset[D]\\|S|=m}}W_S.
\end{equation*}
For $r=0,1,2$, define
\begin{equation*}
M_r
=
\sum_m
m^{\underline r}
\frac{\binom Km}{\binom Dm}F_m,
\end{equation*}
where
\begin{equation*}
m^{\underline r}=m(m-1)\cdots(m-r+1).
\end{equation*}
\begin{lemma}[Marked-incidence identity]\label{lem:marked-incidence}
For $r=0,1,2$,
\begin{equation*}\tag{8.1}\label{eq:marked-incidence}
M_r
=
\frac{r!}{\binom DK}
\sum_{\substack{J\subset I\subset[D]\\|J|=r,\ |I|=K}}
\sum_{J\subset S\subset I}W_S.
\end{equation*}
\end{lemma}

\begin{proof}
Fix a state $S$ with $|S|=m$.  On the right it occurs
\begin{equation*}
\binom mr\binom{D-m}{K-m}
\end{equation*}
times.  Multiplication by $r!/\binom DK$ gives
\begin{equation*}
m^{\underline r}
\frac{\binom Km}{\binom Dm},
\end{equation*}
which is its coefficient on the left.
\end{proof}

For $|J|\le2$, define
\begin{equation*}
U_J
=
\sum_{\substack{I\supset J\\|I|=K}}
\sum_{J\subset S\subset I}W_S.
\end{equation*}
Then
\begin{equation*}
M_0=\binom DK^{-1}U_\varnothing,
\end{equation*}
\begin{equation*}
M_1=\binom DK^{-1}\sum_iU_i,
\end{equation*}
and
\begin{equation*}
M_2
=
2\binom DK^{-1}\sum_{i<j}U_{ij}.
\end{equation*}

\begin{proposition}[Pair/triple incidence decomposition]
\label{prop:pair-triple}
One has
\begin{equation*}\tag{8.2}\label{eq:pair-triple}
\det(M_0,M_1,M_2)
=
\frac{2}{\binom DK^3}
\left(
\sum_{\{a,b\}}\mathcal N_{a,b}
+
\sum_{\{a,b,c\}}\mathcal C_{a,b,c}
\right),
\end{equation*}
where
\begin{equation*}\tag{8.3}\label{eq:Nab}
\mathcal N_{a,b}
=
\det(U_\varnothing,U_a,U_{ab})
+
\det(U_\varnothing,U_b,U_{ab})
\end{equation*}
and
\begin{equation*}\tag{8.4}\label{eq:Cabc}
\begin{aligned}
\mathcal C_{a,b,c}
={}&
\det(U_\varnothing,U_a,U_{bc})
+
\det(U_\varnothing,U_b,U_{ac})\\
&+
\det(U_\varnothing,U_c,U_{ab}).
\end{aligned}
\end{equation*}
\end{proposition}

\begin{proof}
By multilinearity,
\begin{equation*}
\det(M_0,M_1,M_2)
=
\frac{2}{\binom DK^3}
\sum_i\sum_{j<\ell}
\det(U_\varnothing,U_i,U_{j\ell}).
\end{equation*}
If $i\in\{j,\ell\}$, the union of the marked positions has size two and
the term belongs to one $\mathcal N_{a,b}$.  Otherwise the union has size
three and the term belongs to one $\mathcal C_{a,b,c}$.  These two cases
are exhaustive and disjoint, proving \eqref{eq:pair-triple}.
\end{proof}

For the two-position term write
\begin{equation*}
P=W_{00},
\qquad
X=W_{10},
\qquad
Y=W_{01},
\qquad
Z=W_{11}.
\end{equation*}
If the positive active-rank weights are $a_0,a_1,a_2$, then
\begin{equation*}
U_\varnothing=a_0P+a_1X+a_1Y+a_2Z,
\end{equation*}
\begin{equation*}
U_a=a_1X+a_2Z,
\qquad
U_b=a_1Y+a_2Z,
\qquad
U_{ab}=a_2Z.
\end{equation*}

\begin{lemma}[Symmetric two-middle identity]\label{lem:two-middle}
One has
\begin{equation*}\tag{8.5}\label{eq:two-middle}
\mathcal N_{a,b}
=
a_0a_1a_2\det(P,X+Y,Z).
\end{equation*}
\end{lemma}

\begin{proof}
Insert the displayed formulas for $U_\varnothing,U_a,U_b,U_{ab}$ into
\eqref{eq:Nab} and expand by trilinearity.  Terms with repeated columns
vanish.  The only cross terms cancel because
\begin{equation*}
\det(Y,X,Z)+\det(X,Y,Z)=0.
\end{equation*}
The remaining terms give \eqref{eq:two-middle}.
\end{proof}

The sum $X+Y$ is essential.  No sign is asserted for
$\det(P,X,Z)$ or $\det(P,Y,Z)$ separately.  By
Theorem~\ref{thm:mixed-kernel}, the physical coefficient rows occurring in
\eqref{eq:two-middle} have
\begin{equation*}
\mathcal N_{a,b}\ge0.
\end{equation*}

For later use we make the three-position dictionary explicit.  Fix one
fully refined occurrence fiber: permanent skeleton, provenance signature,
spectator labels and placements, retained boundary data, physical
multi-index, and kernel species.  Fix all ownership coordinates except
three distinguished positions $T=\{a,b,c\}$.  Varying the three bits
produces exactly eight states.  By the marked-incidence identity, changing
the mark set only imposes $J\subset S$; it does not change the unmarked
source atom or its scalar.  Hence these eight states are literally the
eight endpoint words of one fixed reset/union coloring fiber from
Theorem~\ref{thm:fixed-union}, with the same chain-product weights.

The same refinement also fixes the entire source-to-physical transport.  If
$R=|\rho|$, the kernel has the form
\begin{equation*}
\Pi_{\mathfrak F}(p,m)
=
\kappa_{\mathfrak F}
\binom pR
\binom{p-R}{m}
B_{\mathfrak F}^{p-R-m},
\qquad \kappa_{\mathfrak F}\ge0.
\end{equation*}
After $q=p-R$, this is a positive row/column scaling of the Pascal matrix
$\binom qm$, and is therefore totally nonnegative of every finite order.
Thus Theorem~\ref{thm:fixed-union} is applied on the source side inside one
fiber and ordinary Cauchy--Binet carries its order-three sign to the physical
coefficient system.  No heterogeneous cross-kernel statement and no generic
``positive convolution preserves minors'' principle is used.

For three active positions, the eight states split into the complementary
pairs
\begin{equation*}
100/011,
\qquad
010/101,
\qquad
001/110.
\end{equation*}
The three determinants in \eqref{eq:Cabc} are exactly this length-three
fixed-union block.  Hence Theorem~\ref{thm:fixed-union} gives
\begin{equation*}\tag{8.6}\label{eq:triple-positive}
\mathcal C_{a,b,c}\ge0.
\end{equation*}
Combining \eqref{eq:pair-triple}, \eqref{eq:two-middle}, and
\eqref{eq:triple-positive} gives the dimension-independent contraction.

\begin{theorem}[Rank-three contraction]\label{thm:rank-three}
For every physical coefficient family obtained after the complete
fixed-union assembly,
\begin{equation*}
\det(M_0,M_1,M_2)\ge0.
\end{equation*}
Only the zero-, one-, and two-mark layers are required, independently of
$D$ and independently of the ambient variable dimension.
\end{theorem}

\begin{remark}[Why the exterior rank stays three]
The ambient dimension enters through the number $D$ of available labeled
sites, but the determinant in Theorem~\ref{thm:rank-three} involves only the
factorial moments of orders $0,1,2$.  Consequently the incidence expansion
can involve at most three distinct marked positions.  This is the precise
sense in which increasing $n$ lengthens the ownership word without creating a
new four-position exterior species.
\end{remark}

The passage from the common coefficient gauge to adjacent physical
coefficient columns is sign preserving.  We record the elementary formula
because it is used in the final assembly.

\begin{lemma}[Positive adjacent-column contraction]
\label{lem:positive-contraction}
Let $u=(u_0,\ldots,u_N)$ and $v=(v_0,\ldots,v_N)$ be two coefficient rows,
and let $\Pi=(\pi_{pm})$ be the binomial coefficient-transport matrix used
to pass from the common gauge to physical coefficient columns.  For two
adjacent physical columns $m,m+1$,
\begin{equation*}\tag{8.7}\label{eq:cauchy-binet-adjacent}
\begin{aligned}
&\det
\begin{pmatrix}
(u\Pi)_m&(u\Pi)_{m+1}\\
(v\Pi)_m&(v\Pi)_{m+1}
\end{pmatrix}\\
&\qquad=
\sum_{0\le p<q\le N}
\det
\begin{pmatrix}
u_p&u_q\\
v_p&v_q
\end{pmatrix}
\det
\begin{pmatrix}
\pi_{pm}&\pi_{p,m+1}\\
\pi_{qm}&\pi_{q,m+1}
\end{pmatrix}.
\end{aligned}
\end{equation*}
For the Pascal/binomial transport
\begin{equation*}
\pi_{pm}=\binom{p}{m},
\end{equation*}
every adjacent $2\times2$ minor on the right of
\eqref{eq:cauchy-binet-adjacent} is nonnegative.
Hence a nonnegative common-gauge projective source remains nonnegative
after the physical coefficient extraction.
\end{lemma}

\begin{proof}
Equation~\eqref{eq:cauchy-binet-adjacent} is the $2\times2$ Cauchy--Binet
formula.  If $p<q$ and $p,q\ge m$, then
\begin{equation*}
\begin{aligned}
&\binom{p}{m}\binom{q}{m+1}
-\binom{p}{m+1}\binom{q}{m}\\
&\qquad=
\binom{p}{m}\binom{q}{m}
\frac{q-p}{m+1}\ge0.
\end{aligned}
\end{equation*}
If one of the binomial coefficients is outside its natural range, the same
minor is obtained by setting that coefficient equal to zero and is again
nonnegative.  Thus every weight in the Cauchy--Binet sum is nonnegative.
\end{proof}

\section{Assembly of the order-raising residual}
\label{sec:assembly}

We return to the residual \eqref{eq:residual} and connect it to the finite
coefficient identities of Sections~\ref{sec:two-chain}--\ref{sec:marked}.  Fix
\begin{equation*}
2\le k\le n-2.
\end{equation*}

Fix maximum-length saturated unit-cover chains
\begin{equation*}
\lambda=\nu^{(0)}\succ\nu^{(1)}\succ\cdots\succ\nu^{(p)}=\gamma
\end{equation*}
and
\begin{equation*}
\gamma=\eta^{(0)}\succ\eta^{(1)}\succ\cdots\succ\eta^{(q)}=\mu.
\end{equation*}
Such chains exist because the dominance intervals are finite.  The exact
telescoping identities below hold for every saturated chain; choosing a
maximum-length representative is used only for the separated-cover
normalization in Lemma~\ref{lem:post-turn-injection}.
For an upper-chain cell $e=1,\ldots,p$ define
\begin{equation*}
U_{e,r}
=
\Ical_r^{\nu^{(e-1)}}-\Ical_r^{\nu^{(e)}},
\qquad
u_e(z)=U_{e,1}(z),
\end{equation*}
and for a lower-chain cell $f=1,\ldots,q$ define
\begin{equation*}
L_{f,r}
=
\Ical_r^{\eta^{(f-1)}}-\Ical_r^{\eta^{(f)}},
\qquad
\ell_f(z)=L_{f,1}(z).
\end{equation*}
Telescoping gives the exact identities
\begin{equation*}\tag{9.1}\label{eq:cell-telescope}
A_r=\sum_{e=1}^pU_{e,r},
\qquad
B_r=\sum_{f=1}^qL_{f,r},
\qquad
C_r=\sum_eU_{e,r}-\sum_fL_{f,r}.
\end{equation*}
For one unit cover
\begin{equation*}
(a,b)\longmapsto(a-1,b+1),
\end{equation*}
the rank-one difference is
\begin{equation*}\tag{9.2}\label{eq:cell-onevar}
z^a+z^b-z^{a-1}-z^{b+1}
=
(z-1)^2\sum_{h=b}^{a-2}z^h.
\end{equation*}
Hence every cover cell carries a finite interval of permanently labeled
monomial heights.

For a rank-$k$ function $F$ define
\begin{equation*}
\mathsf T_zF(x_1,\ldots,x_k)
=
\sum_{j=1}^k
F(x_1,\ldots,x_jz,\ldots,x_k).
\end{equation*}
For an upper cell put
\begin{equation*}
\Phi_e^\gamma
=
S_\gamma(z)U_{e,k}
+\Ical_k^\gamma u_e(z)
-\mathsf T_zU_{e,k},
\end{equation*}
and for a lower cell put
\begin{equation*}
\Psi_f^\gamma
=
S_\gamma(z)L_{f,k}
+\Ical_k^\gamma \ell_f(z)
-\mathsf T_zL_{f,k}.
\end{equation*}
Substitution of \eqref{eq:cell-telescope} into
\eqref{eq:residual} gives
\begin{equation*}\tag{9.3}\label{eq:residual-cell}
\boxed{\begin{aligned}
\mathcal R_k
={}&
\sum_{e=1}^p\Phi_e^\gamma
-
\sum_{f=1}^q\Psi_f^\gamma\\
&+
\sum_{e,e'=1}^p
u_e(z)U_{e',k}
+
\sum_{f,f'=1}^q
\ell_f(z)L_{f',k}.
\end{aligned}}
\end{equation*}
This is an exact algebraic identity.  No sign is asserted for the
individual $\Phi_e^\gamma$ or $\Psi_f^\gamma$.

We next separate the unchanged exponent labels from the labels involved in
one local cover interaction.  Let $\tau$ denote such an interaction and let
$A_\tau\subset[n]$ be the set of exponent labels whose values change in
$\tau$.  A term of the injective sum uses only $k$ exponent labels, so the
spectator data consist not of all labels outside $A_\tau$, but of the
selected labels outside $A_\tau$ together with their physical positions.
More precisely, let
\begin{equation*}
R_\tau\subset[n]\setminus A_\tau
\end{equation*}
be the selected spectator labels occurring in one term, and let
\begin{equation*}
\chi:R_\tau\longrightarrow\{1,\ldots,k\}
\end{equation*}
record the distinct physical variable positions receiving those labels.
The datum $(\tau,R_\tau,\chi)$ will be called a \emph{spectator thread}.
Since no exponent carried by a label in $R_\tau$ changes during the local
interaction, every term on this thread contains the same positive monomial
factor
\begin{equation*}
 m_\chi(\mathbf x)
 =
 \prod_{r\in R_\tau}x_{\chi(r)}^{\alpha_r}>0,
\end{equation*}
where $\alpha_r$ is the unchanged exponent on label $r$.  After factoring
out $m_\chi$, the remaining expression contains only the active labels and
is expanded in the common ordered coefficient basis used in
Sections~\ref{sec:fixed-union}--\ref{sec:marked}.

Accordingly, after the terms in \eqref{eq:residual-cell} are partitioned by
their local interaction and spectator data, the residual has the exact form
\begin{equation*}
\mathcal R_k
=
\sum_\chi m_\chi(\mathbf x)\,\mathcal S_\chi,
\qquad
m_\chi(\mathbf x)>0.
\end{equation*}

The next lemma identifies the coefficient rows that will be compared
projectively.

\begin{lemma}[Physical positive-basis expansion]
\label{lem:physical-positive-basis}
Fix an ordered chamber and a spectator thread $\chi$.  For one local
rank-three realization order its active variables as
\begin{equation*}
a\ge b\ge c>0
\end{equation*}
and introduce the nonnegative normal coordinates
\begin{equation*}
 a=c+v+u,
 \qquad
 b=c+v,
 \qquad
 c>0,
 \qquad
 u,v\ge0.
\end{equation*}
If additional active variables occur on the same thread, order them as well
and denote their successive nonnegative gaps by
\begin{equation*}
w_1,\ldots,w_d\ge0.
\end{equation*}
The variables and positive factors that are not expanded in this local
coefficient extraction are collected into a boundary-parameter vector
$\mathbf b$.  In particular, boundary factors that will later be kept
pointwise, such as values of $C_k$ or $C_1$, are not expanded into scalar
coefficients.

Fix all exponents of
\begin{equation*}
c,
\quad v,
\quad w_1,\ldots,w_d
\end{equation*}
and collect them into
\begin{equation*}
\rho=(r_0,r_1,r_2,\ldots,r_{d+1})\in\mathbb Z_{\ge0}^{d+2}.
\end{equation*}
The remaining exponent of the distinguished normal coordinate $u$ is the
projective-column index $m$.  Define the physical basis function
\begin{equation*}\tag{9.4a}\label{eq:positive-physical-basis}
\boxed{
\Phi_{\chi,\rho,m}(\mathbf x,z)
=
c^{r_0}v^{r_1}
\prod_{j=1}^{d}w_j^{r_{j+1}}
 u^m.
}
\end{equation*}
Thus
\begin{equation*}
\Phi_{\chi,\rho,m}(\mathbf x,z)\ge0
\end{equation*}
on the ordered chamber.

There are coefficient functions
\begin{equation*}
p_{\chi,\rho,m}(\mathbf b),
\qquad
0\le m\le M_{\chi,\rho},
\end{equation*}
such that
\begin{equation*}\tag{9.4b}\label{eq:thread-positive-expansion}
\boxed{
\mathcal S_\chi
=
\sum_\rho
\sum_{m=0}^{M_{\chi,\rho}}
p_{\chi,\rho,m}(\mathbf b)\,
\Phi_{\chi,\rho,m}(\mathbf x,z).
}
\end{equation*}
For fixed $(\chi,\rho,\mathbf b)$, write
\begin{equation*}
P_{\chi,\rho}(\mathbf b)
=
\bigl(
 p_{\chi,\rho,0}(\mathbf b),\ldots,
 p_{\chi,\rho,M_{\chi,\rho}}(\mathbf b)
\bigr).
\end{equation*}
This is exactly the physical coefficient row obtained from the common
source by the Pascal/binomial transport of
Lemma~\ref{lem:positive-contraction}.  Consequently,
\begin{equation*}\tag{9.4c}\label{eq:residual-positive-expansion}
\boxed{
\mathcal R_k(\mathbf x,z)
=
\sum_{\chi,\rho}
m_\chi(\mathbf x)
\sum_{m=0}^{M_{\chi,\rho}}
p_{\chi,\rho,m}(\mathbf b)\,
\Phi_{\chi,\rho,m}(\mathbf x,z).
}
\end{equation*}
\end{lemma}

\begin{proof}
The construction is a finite coefficient extraction in nonnegative normal
coordinates.  After the spectator factor $m_\chi$ has been removed, every
active physical monomial is first written in the ordered coordinates above.
For example, a shifted active power has the form
\begin{equation*}
\bigl(T_0(\mathbf b,c,v,\mathbf w)+u\bigr)^p,
\end{equation*}
where $T_0$ is a sum of retained positive boundary variables and
nonnegative normal coordinates.  Expanding only the distinguished normal
coordinate gives
\begin{equation*}\tag{9.4d}\label{eq:pascal-u-expansion}
\bigl(T_0+u\bigr)^p
=
\sum_{m=0}^{p}
\binom pm T_0^{p-m}u^m.
\end{equation*}
The coefficients $\binom pm$ are nonnegative.  Next expand the powers
$T_0^{p-m}$ in the auxiliary normal coordinates
$c,v,w_1,\ldots,w_d$, while leaving the boundary parameters $\mathbf b$
untouched.  The multinomial theorem gives an exact finite expansion
\begin{equation*}\tag{9.4e}\label{eq:physical-transport-expanded}
\bigl(T_0+u\bigr)^p
=
\sum_{\rho,m}
\Pi_{\chi,\rho}(p,m;\mathbf b)\,
\Phi_{\chi,\rho,m}(\mathbf x,z),
\end{equation*}
where
\begin{equation*}
\Pi_{\chi,\rho}(p,m;\mathbf b)\ge0
\end{equation*}
pointwise on the ordered chamber.  Each
$\Pi_{\chi,\rho}(p,m;\mathbf b)$ is a product of a binomial coefficient,
multinomial coefficients, and nonnegative powers of retained positive
boundary quantities.

We call the exponent $p$ in a pre-expansion term $(T_0+u)^p$ its
\emph{common-source position}.  Let $q_{\chi,p}(\mathbf b)$ denote the
signed coefficient of common-source position $p$ after all indices not
involved in this row have been fixed.  Collecting equal basis monomials in
\eqref{eq:physical-transport-expanded} gives the exact coefficient formula
\begin{equation*}\tag{9.4f}\label{eq:p-physical-coefficient}
\boxed{
p_{\chi,\rho,m}(\mathbf b)
=
\sum_p
q_{\chi,p}(\mathbf b)\,
\Pi_{\chi,\rho}(p,m;\mathbf b).
}
\end{equation*}
Thus the passage from the common source row to the physical row is the
Pascal matrix
\begin{equation*}
\pi_{pm}=\binom pm
\end{equation*}
followed by positive diagonal factors and the specific Pascal/multinomial
coefficient transports coming from the auxiliary expansions; these transports
are totally nonnegative in the ordered coefficient indices.  This is precisely the transport
whose adjacent $2\times2$ minors are handled by
Lemma~\ref{lem:positive-contraction}.

Equation~\eqref{eq:thread-positive-expansion} is now simply the collection
of all physical monomials with the same fixed spectator thread.  Summing
the exact thread decomposition
\begin{equation*}
\mathcal R_k
=
\sum_\chi m_\chi\mathcal S_\chi
\end{equation*}
over $\chi$ gives
\eqref{eq:residual-positive-expansion}.
\end{proof}

The three factors in \eqref{eq:residual-positive-expansion} now have
separate roles.  The spectator factor satisfies
\begin{equation*}
m_\chi(\mathbf x)>0,
\end{equation*}
and the physical basis function satisfies
\begin{equation*}
\Phi_{\chi,\rho,m}(\mathbf x,z)\ge0.
\end{equation*}
The only nontrivial sign is therefore the coefficient function
\begin{equation*}
p_{\chi,\rho,m}(\mathbf b).
\end{equation*}
The projective-ordering argument below is applied pointwise in the retained
boundary parameters $\mathbf b$.  If every coefficient function in every
row is nonnegative, then
\eqref{eq:residual-positive-expansion} implies
\begin{equation*}
\mathcal R_k(\mathbf x,z)\ge0.
\end{equation*}

For a fixed thread, write
\begin{equation*}
\mathcal S_\chi
=
\mathcal D_\chi+\mathcal G_\chi,
\end{equation*}
where the two pieces are obtained directly from the two lines of
\eqref{eq:residual-cell}:
\begin{equation*}
\mathcal D_\chi
=
\left[
\sum_{e=1}^p\Phi_e^\gamma
-
\sum_{f=1}^q\Psi_f^\gamma
\right]_\chi,
\end{equation*}
\begin{equation*}
\mathcal G_\chi
=
\left[
\sum_{e,e'=1}^p u_e(z)U_{e',k}
+
\sum_{f,f'=1}^q\ell_f(z)L_{f',k}
\right]_\chi.
\end{equation*}
Here $[\,\cdot\,]_\chi$ means: retain the terms with the fixed spectator
data $\chi$, remove the common factor $m_\chi$, and write the remaining
active expression in the common ordered coefficient basis.  By
\eqref{eq:cell-telescope}, the second display is exactly the threadwise
part of
\begin{equation*}
A_1(z)A_k(\mathbf x)+B_1(z)B_k(\mathbf x)
\end{equation*}
from the original residual \eqref{eq:residual}.

The signed term can also be written directly in terms of the original
rank defects.  Indeed, by \eqref{eq:cell-telescope},
\begin{equation*}\tag{9.4}\label{eq:Dchi-original}
\boxed{
\mathcal D_\chi
=
\left[
S_\gamma(z)C_k(\mathbf x)
+\Ical_k^\gamma(\mathbf x)C_1(z)
-\sum_{j=1}^k C_k(x_1,\ldots,x_jz,\ldots,x_k)
\right]_\chi .
}
\end{equation*}
The next step is not a further identity for the raw polynomial
$\mathcal D_\chi$.  The local source decomposition becomes exact after one
passes to the common source representing an adjacent physical projective
bracket.  We now make that passage explicit.

Fix a spectator thread $\chi$, fix one of the remaining coefficient
multi-indices $\rho$ in \eqref{eq:thread-positive-expansion}, and fix an
adjacent pair of physical projective columns $m,m+1$.  Put
\begin{equation*}
P_{\chi,\rho}
=
(p_{\chi,\rho,0},\ldots,p_{\chi,\rho,M}),
\qquad
H_{\chi,\rho}
=
(h_{\chi,\rho,0},\ldots,h_{\chi,\rho,M}),
\end{equation*}
where $M=M_{\chi,\rho}$.  The first row is exactly the coefficient row in
\eqref{eq:thread-positive-expansion}; the second is the positive reference
row produced by the unmarked common source in the same physical basis.
For a row $F=(f_0,\ldots,f_M)$ define its adjacent bracket with
$H_{\chi,\rho}$ by
\begin{equation*}\tag{9.5}\label{eq:thread-bracket}
\mathfrak B_{\chi,\rho,m}(F)
=
f_mh_{\chi,\rho,m+1}-f_{m+1}h_{\chi,\rho,m}.
\end{equation*}
Because $H_{\chi,\rho}$ is fixed,
$\mathfrak B_{\chi,\rho,m}$ is linear in $F$.

Sections~\ref{sec:fixed-union} and~\ref{sec:marked} give a canonical
common-source representative for this bracket before the final positive
Pascal/Cauchy--Binet contraction.  Denote by
\begin{equation*}
\mathcal D^{\mathrm{src}}_{\chi,\rho,m}
\end{equation*}
the common-source representative of the contribution of
$\mathcal D_\chi$, and by
\begin{equation*}
\mathcal G^{\mathrm{src}}_{\chi,\rho,m}
\end{equation*}
the representative of the contribution of $\mathcal G_\chi$.
Thus the total source for this thread and adjacent column pair is
\begin{equation*}\tag{9.6}\label{eq:thread-source-split}
\mathcal S^{\mathrm{src}}_{\chi,\rho,m}
=
\mathcal D^{\mathrm{src}}_{\chi,\rho,m}
+
\mathcal G^{\mathrm{src}}_{\chi,\rho,m}.
\end{equation*}
The positive contraction $\Lambda_{\chi,\rho,m}$ of
Lemma~\ref{lem:positive-contraction} satisfies
\begin{equation*}\tag{9.7}\label{eq:source-to-bracket}
\Lambda_{\chi,\rho,m}
\bigl(\mathcal S^{\mathrm{src}}_{\chi,\rho,m}\bigr)
=
\mathfrak B_{\chi,\rho,m}(P_{\chi,\rho}).
\end{equation*}

We next write $\mathcal D^{\mathrm{src}}_{\chi,\rho,m}$ explicitly.  In the
adjacent subset-rank expansion of Section~\ref{sec:fixed-union}, let
$P_S,Q_S$ denote the two subset-indexed source families entering the
bracket.  Set
\begin{equation*}
P_k=\sum_{|S|=k}P_S,
\qquad
Q_k=\sum_{|S|=k}Q_S.
\end{equation*}
For $|S|=k$ and $|T|=k+1$, write
\begin{equation*}
C=S\cap T,
\qquad
U=S\triangle T.
\end{equation*}
If $|C|=k-r$, then $|U|=2r+1$ and uniquely
\begin{equation*}
S=C\cup A,
\qquad
T=C\cup(U\setminus A),
\qquad
|A|=r.
\end{equation*}
Hence the exact adjacent-rank reindexing gives
\begin{equation*}\tag{9.8}\label{eq:D-source-fixed-union}
\boxed{
\mathcal D^{\mathrm{src}}_{\chi,\rho,m}
=
\sum_{r=0}^k
\sum_{\substack{C\cap U=\varnothing\\|C|=k-r,\ |U|=2r+1}}
\widehat{\mathcal B}_{\chi,\rho,m;C,U},
}
\end{equation*}
where
\begin{equation*}\tag{9.9}\label{eq:thread-fixed-block}
\widehat{\mathcal B}_{\chi,\rho,m;C,U}
=
\sum_{\substack{A\subset U\\|A|=r}}
\left(
P_{C\cup A}Q_{C\cup(U\setminus A)}
-
P_{C\cup(U\setminus A)}Q_{C\cup A}
\right),
\end{equation*}
after the common positive rank and spectator factors have been removed.
Every term in the adjacent bracket occurs in exactly one block
\eqref{eq:thread-fixed-block}.

The fixed-union exchange theorem of
Section~\ref{sec:fixed-union} straightens each complete fixed-$(C,U)$ block
subtraction-free into terminal rank-three source blocks.  Let
\begin{equation*}
\mathfrak J_{\chi,\rho,m}
\end{equation*}
be the finite set of all such terminal blocks produced from all
$(C,U)$ in \eqref{eq:D-source-fixed-union}.  The straightening gives the
exact source identity
\begin{equation*}\tag{9.10}\label{eq:D-source-terminal}
\boxed{
\mathcal D^{\mathrm{src}}_{\chi,\rho,m}
=
\mathcal U_{\chi,\rho,m}
+
\mathcal H_{\chi,\rho,m}
+
\sum_{\tau\in\mathfrak J_{\chi,\rho,m}}
\mathcal Q_{\chi,\rho,m;\tau}.
}
\end{equation*}
Here $\mathcal U_{\chi,\rho,m}$ is the unmatched endpoint source and
$\mathcal H_{\chi,\rho,m}$ is the second-and-higher normal remainder left after
the wall and first-normal terms have been extracted.

The set $\mathfrak J_{\chi,\rho,m}$ has a disjoint and exhaustive partition
\begin{equation*}\tag{9.11}\label{eq:J-partition}
\mathfrak J_{\chi,\rho,m}
=
\mathfrak J^A_{\chi,\rho,m}
\sqcup
\mathfrak J^C_{\chi,\rho,m}
\sqcup
\mathfrak J^Z_{\chi,\rho,m}
\sqcup
\mathfrak J^P_{\chi,\rho,m},
\end{equation*}
according to the supports and orientations of the two labeled covers in
the terminal block:
\begin{equation*}
\begin{array}{c|c}
\mathfrak J^A_{\chi,\rho,m}
&\text{same-role shared-label interaction},\\
\mathfrak J^C_{\chi,\rho,m}
&\text{separated commuting interaction},\\
\mathfrak J^Z_{\chi,\rho,m}
&\text{nested equal-pair degeneration},\\
\mathfrak J^P_{\chi,\rho,m}
&\text{donor/receiver role switch}.
\end{array}
\end{equation*}
Define
\begin{equation*}
\mathcal A_{\chi,\rho,m}
=
\sum_{\tau\in\mathfrak J^A_{\chi,\rho,m}}
\mathcal Q_{\chi,\rho,m;\tau},
\qquad
\mathcal C_{\chi,\rho,m}
=
\sum_{\tau\in\mathfrak J^C_{\chi,\rho,m}}
\mathcal Q_{\chi,\rho,m;\tau},
\end{equation*}
\begin{equation*}
\mathcal Z_{\chi,\rho,m}
=
\sum_{\tau\in\mathfrak J^Z_{\chi,\rho,m}}
\mathcal Q_{\chi,\rho,m;\tau}.
\end{equation*}
For a role-switch block the canonical source orientation is negative, so
write
\begin{equation*}
\mathcal Q_{\chi,\rho,m;\tau}
=-\mathcal P_{\chi,\rho,m;\tau}
\qquad
(\tau\in\mathfrak J^P_{\chi,\rho,m})
\end{equation*}
and define
\begin{equation*}
\mathcal P_{\chi,\rho,m}
=
\sum_{\tau\in\mathfrak J^P_{\chi,\rho,m}}
\mathcal P_{\chi,\rho,m;\tau}.
\end{equation*}
Substitution of the partition \eqref{eq:J-partition} into
\eqref{eq:D-source-terminal} gives the exact bracket-source decomposition
\begin{equation*}\tag{9.12}\label{eq:source-decomposition}
\boxed{
\mathcal D^{\mathrm{src}}_{\chi,\rho,m}
=
\mathcal A_{\chi,\rho,m}
+
\mathcal C_{\chi,\rho,m}
+
\mathcal Z_{\chi,\rho,m}
+
\mathcal H_{\chi,\rho,m}
+
\mathcal U_{\chi,\rho,m}
-
\mathcal P_{\chi,\rho,m}.
}
\end{equation*}
This is the level at which the local source classification is used.
No corresponding termwise decomposition of the raw polynomial
$\mathcal D_\chi$ is asserted.

We now prove directly the local three-label statement needed for the
same-role class.  This is an elementary $A_2$ rhombus calculation and uses no
three-variable collision theorem.

\begin{proposition}[Self-contained same-role three-label rhombus theorem]
\label{prop:three-label-input}
Let $A\ge B\ge C\ge0$ be integers.  For the same-donor orientation assume
$A-B\ge3$, which is exactly the condition that both successive unit transfers
from the permanently labeled $A$-part are legal while the exponent order is
retained.  For the same-receiver orientation assume $B-C\ge3$, the analogous
condition that two successive units can enter the permanently labeled $C$-part.
Define the same-donor rhombus
\begin{equation*}
\begin{aligned}
\mathscr D_{A,B,C}
={}&J^{(3)}_{(A,B,C)}
-J^{(3)}_{(A-1,B+1,C)}
-J^{(3)}_{(A-1,B,C+1)}\\
&+J^{(3)}_{(A-2,B+1,C+1)},
\end{aligned}
\end{equation*}
and the same-receiver rhombus
\begin{equation*}
\begin{aligned}
\mathscr R_{A,B,C}
={}&J^{(3)}_{(A,B,C)}
-J^{(3)}_{(A-1,B,C+1)}
-J^{(3)}_{(A,B-1,C+1)}\\
&+J^{(3)}_{(A-1,B-1,C+2)}.
\end{aligned}
\end{equation*}
On the ordered chamber
\begin{equation*}
 x=\zeta+v+u,\qquad y=\zeta+v,\qquad z=\zeta,
 \qquad u,v,\zeta\ge0,
\end{equation*}
both rhombi belong to $\R_{\ge0}[u,v,\zeta]$.  Consequently every complete
same-role shared-label packet, after the common positive spectator and
transport factors are restored, is coefficientwise nonnegative in the
ordered-gap basis.
\end{proposition}

\begin{proof}
We treat the donor and receiver orientations separately.  Put
\begin{equation*}
 w=xyz,
 \qquad
 \Delta=(x-y)(x-z)(y-z),
\end{equation*}
and, for $r\ge0$,
\begin{equation*}
 S_r=\sum_{\rm cyc}x^r(x-y)(x-z).
\end{equation*}
For the donor rhombus, with
\begin{equation*}
 d=B-C,
 \qquad
 \rho=A-C-2,
\end{equation*}
a direct expansion gives
\begin{equation*}\tag{9.13a}\label{eq:donor-cyclic}
\mathscr D_{A,B,C}
=w^C\bigl(p_dS_\rho-S_{\rho+d}\bigr),
\qquad p_d=x^d+y^d+z^d.
\end{equation*}
If $d\ge2$, the neighboring-rhombus difference is the bialternant identity
\begin{equation*}\tag{9.13b}\label{eq:donor-equalization}
\boxed{
\mathscr D_{A,B,C}-\mathscr D_{A,B-1,C+1}
=\Delta^2 s_{(A-4,B-2,C)}(x,y,z).
}
\end{equation*}
The partition on the right is legal because the two-step same-donor packet
has $A-B\ge3$, while the equalization step assumes $B-C\ge2$.  Since Schur
polynomials have nonnegative monomial coefficients, the right side of
\eqref{eq:donor-equalization} is coefficientwise nonnegative after the
ordered-gap substitution.  Thus it remains only to check $d=0,1$.

In the ordered-gap variables one has
\begin{equation*}\tag{9.13c}\label{eq:S-gap-positive}
S_r
=u^2x^r
+uv\bigl(x^r-y^r+z^r\bigr)
+v^2z^r
\in\R_{\ge0}[u,v,\zeta].
\end{equation*}
Hence the $d=0$ base
\begin{equation*}
\mathscr D_{A,C,C}=2w^C S_{A-C-2}
\end{equation*}
is coefficientwise nonnegative.  For $d=1$, put
$e_1=x+y+z$.  Then
\begin{equation*}
\mathscr D_{A,C+1,C}
=w^C\bigl(e_1S_\rho-S_{\rho+1}\bigr),
\end{equation*}
and
\begin{equation*}\tag{9.13d}\label{eq:first-offdiag-gap}
 e_1S_r-S_{r+1}
 =u^2x^r(y+z)+uvM_r+v^2z^r(x+y),
\end{equation*}
where $M_0=2y$, while for $r\ge1$
\begin{equation*}
\begin{aligned}
M_r
={}&2yz^r
+u\left((y+z)\sum_{j=0}^{r-1}x^{r-1-j}y^j-y^r+z^r\right).
\end{aligned}
\end{equation*}
The expression in parentheses is coefficientwise nonnegative, because
\begin{equation*}
\sum_{j=0}^{r-1}x^{r-1-j}y^j-r y^{r-1}
\in\R_{\ge0}[u,y]
\end{equation*}
and therefore
\begin{equation*}
\begin{aligned}
&(y+z)\sum_{j=0}^{r-1}x^{r-1-j}y^j-y^r+z^r\\
&\quad=(y+z)
\left(\sum_{j=0}^{r-1}x^{r-1-j}y^j-r y^{r-1}\right)
 +(r-1)y^r+rzy^{r-1}+z^r.
\end{aligned}
\end{equation*}
This proves the donor orientation.

For the receiver orientation, the same bialternant calculation, equivalently
applied to complementary exponent vectors and then cleared of the common
monomial, gives the upper-gap equalization
\begin{equation*}\tag{9.13e}\label{eq:receiver-equalization}
\boxed{
\mathscr R_{A,B,C}-\mathscr R_{A-1,B+1,C}
=\Delta^2 s_{(A-4,B-2,C)}(x,y,z)
}
\end{equation*}
whenever $A-B\ge2$.  Thus only the bases $A-B=0,1$ remain.  Define
\begin{equation*}
T_r=\sum_{\rm cyc}(yz)^r(x-y)(x-z).
\end{equation*}
A direct ordered-gap expansion is
\begin{equation*}\tag{9.13f}\label{eq:T-gap-positive}
T_r
=u^2(yz)^r
+uv\bigl((xy)^r-(xz)^r+(yz)^r\bigr)
+v^2(xy)^r,
\end{equation*}
which is coefficientwise nonnegative because
$(xy)^r-(xz)^r=x^r(y^r-z^r)$.  Hence
\begin{equation*}
\mathscr R_{A,A,C}=2w^C T_{A-C-1}
\end{equation*}
is nonnegative coefficientwise.

For $A=B+1$, write $s=B-C-1$.  Under the two-step same-receiver
legality hypothesis $B-C\ge3$ one has $s\ge2$ (the displayed algebra in fact
remains valid already for $s\ge1$).  Direct expansion gives
\begin{equation*}
\mathscr R_{B+1,B,C}=w^C F_s,
\end{equation*}
with
\begin{equation*}\tag{9.13g}\label{eq:receiver-base-one}
F_s
=u^2(yz)^s(y+z)+uvN_s+v^2(xy)^s(x+y),
\end{equation*}
where
\begin{equation*}
\begin{aligned}
N_s
&=(xy)^s(x+y)-(xz)^s(x+z)+(yz)^s(y+z)\\
&=x^s\bigl(x(y^s-z^s)+(y^{s+1}-z^{s+1})\bigr)
 +(yz)^s(y+z).
\end{aligned}
\end{equation*}
Every term is coefficientwise nonnegative in $u,v,\zeta$.  This closes the
receiver orientation and hence both same-role packets.  Reattaching frozen
spectators only multiplies these local polynomials by positive monomials and
sums them with nonnegative incidence and transport coefficients, so the sign
is preserved.
\end{proof}

The first three local classes now have their signs.  Proposition
\ref{prop:three-label-input} gives
\begin{equation*}\tag{9.13}\label{eq:Apositive}
\mathcal A_{\chi,\rho,m}\ge0.
\end{equation*}
Separated covers commute and give a positive commuting source, so
\begin{equation*}\tag{9.14}\label{eq:Cpositive}
\mathcal C_{\chi,\rho,m}\ge0.
\end{equation*}
For an equal-pair degeneration,
\begin{equation*}\tag{9.15}\label{eq:Zzero}
\mathcal Z_{\chi,\rho,m}=0.
\end{equation*}

It remains to compare the role-switch source with the attached Green
source.  For one labeled pair of cover-cell heights $\ell<h$, set
\begin{equation*}
d=h-\ell.
\end{equation*}
The normalized Green term is
\begin{equation*}\tag{9.16}\label{eq:green-source-nine}
\mathcal G_{\ell,h}(X,Y)
=
\frac{(X^h-X^\ell)(Y^\ell-Y^h)}
{(X-1)(1-Y)}
=
(XY)^\ell A_d(X)A_d(Y),
\end{equation*}
where
\begin{equation*}
A_d(T)=1+T+\cdots+T^{d-1}.
\end{equation*}
The corresponding symmetric pair-first source is
\begin{equation*}\tag{9.17}\label{eq:pair-source-nine}
\mathcal P_{\ell,h}(X,Y)
=
\frac{X^hY^\ell-X^\ell Y^h}{X-Y}
=
(XY)^\ell
\sum_{q=0}^{d-1}X^{d-1-q}Y^q.
\end{equation*}
The second polynomial is one anti-diagonal of the full $d\times d$
coefficient rectangle in the first.  Therefore
\begin{equation*}\tag{9.18}\label{eq:green-minus-pair}
\boxed{
\mathcal G_{\ell,h}-\mathcal P_{\ell,h}
=
(XY)^\ell
\sum_{\substack{0\le i,j\le d-1\\i+j\ne d-1}}
X^iY^j
\ge_{\mathrm{coeff}}0.
}
\end{equation*}
Distinct role switches use disjoint sets of permanently labeled unordered
cover-cell pairs.  Hence no Green pair is used twice.  Allocating the
corresponding part of the positive product source to the fixed thread and
adjacent column pair, and summing \eqref{eq:green-minus-pair}, gives
\begin{equation*}\tag{9.19}\label{eq:thread-green}
\boxed{
\mathcal G^{\mathrm{src}}_{\chi,\rho,m}-\mathcal P_{\chi,\rho,m}
\ge_{\mathrm{coeff}}0.
}
\end{equation*}

We now record the only remaining local-normal term.  For $r\ge0$ put
\begin{equation*}
h_r(a,b)=\sum_{j=0}^{r}a^{r-j}b^j.
\end{equation*}
If $\{a,b\}$ is one of the three unordered pairs from $\{x,y,z\}$, let
$s_{a,b}$ denote the remaining variable.  For $m\ge2$ and $S\ge1$ define
\begin{equation*}\tag{9.20}\label{eq:weighted-corner}
\begin{aligned}
\mathcal C^{(3)}_{m,S}(x,y,z)
={}&
\sum_{\{a,b\}\subset\{x,y,z\}}
(ab)^{S-1}(a-b)^2\\
&\quad\times
\left((m-1)h_{m-1}(a,b)
-ms_{a,b}h_{m-2}(a,b)\right).
\end{aligned}
\end{equation*}

\begin{lemma}[Positive higher-normal remainder]\label{lem:higher-normal}
For $m\ge2$ and $S\ge1$,
\begin{equation*}\tag{9.21}\label{eq:higher-normal}
\mathcal C^{(3)}_{m,S}(c+v+u,c+v,c)
=
B_{m,S}(v,c)
+\frac{u}{2}\partial_vB_{m,S}(v,c)
+u^2H_{m,S}(u,v,c),
\end{equation*}
where
\begin{equation*}
H_{m,S}\in\R_{\ge0}[u,v,c].
\end{equation*}
Here
\begin{equation*}
B_{m,S}(v,c)
=
-2[c(c+v)]^{S-1}v^3L_m(c+v,c)
\end{equation*}
and
\begin{equation*}
L_m(B,C)
=
\sum_{r=0}^{m-2}(r+1)B^{m-2-r}C^r.
\end{equation*}
The same conclusion holds after adjoining arbitrary spectator variables by
the fixed-variable branching identity.
\end{lemma}

\begin{proof}
Write
\begin{equation*}
U=u,\qquad V=v,\qquad C=c,\qquad
A=C+V+U,\qquad B=C+V,
\qquad q=S-1\ge0.
\end{equation*}
For an unordered pair $\{X,Y\}$ and the remaining variable $R$, define
\begin{equation*}
\mathcal T(X,Y;R)
=(XY)^q(X-Y)
\left[(m-1)(X^m-Y^m)-mR(X^{m-1}-Y^{m-1})\right].
\end{equation*}
Because $(X^m-Y^m)/(X-Y)=h_{m-1}(X,Y)$, the definition
\eqref{eq:weighted-corner} is exactly
\begin{equation*}\tag{9.21a}\label{eq:corner-three-pairs}
\mathcal C^{(3)}_{m,S}
=\mathcal T(A,B;C)+\mathcal T(A,C;B)+\mathcal T(B,C;A).
\end{equation*}
The last term is affine in $U$, since $A=B+U$ occurs there only as the
remaining-variable factor.  It therefore contributes nothing to the
second-and-higher $U$-Taylor remainder.

For the top pair,
\begin{equation*}\tag{9.21b}\label{eq:top-pair-factor}
\mathcal T(A,B;C)=U^2(AB)^qK_m,
\end{equation*}
where
\begin{equation*}
K_m
=(m-1)\sum_{r=0}^{m-1}A^{m-1-r}B^r
-mC\sum_{r=0}^{m-2}A^{m-2-r}B^r.
\end{equation*}
Expanding only in $U=A-B$ gives
\begin{equation*}\tag{9.21c}\label{eq:Km-positive}
\boxed{
\begin{aligned}
K_m={}&m(m-1)VB^{m-2}+(m-1)U^{m-1}\\
&+\sum_{i=1}^{m-2}\binom{m}{i+1}
U^iB^{m-2-i}\bigl(iC+(m-1)V\bigr).
\end{aligned}}
\end{equation*}
Thus $K_m\in\R_{\ge0}[U,V,C]$.  With
\begin{equation*}
L_m(B,C)=\sum_{r=0}^{m-2}(r+1)B^{m-2-r}C^r,
\end{equation*}
we have, since $B=C+V$,
\begin{equation*}
L_m\le_{\rm coeff}\binom m2 B^{m-2},
\end{equation*}
and the constant-in-$U$ term of \eqref{eq:Km-positive} yields the stronger
reserve
\begin{equation*}\tag{9.21d}\label{eq:Km-reserve}
\boxed{K_m\ge_{\rm coeff}2V L_m.}
\end{equation*}

For the side pair set
\begin{equation*}
\beta_m(X)
=(m-1)(X^m-C^m)-mB(X^{m-1}-C^{m-1}),
\qquad
f_q(X)=X^q(X-C).
\end{equation*}
Then
\begin{equation*}
\mathcal T(A,C;B)=C^q f_q(A)\beta_m(A).
\end{equation*}
At $X=B$,
\begin{equation*}\tag{9.21e}\label{eq:beta-wall}
\beta_m(B)=-V^2L_m(B,C).
\end{equation*}
Moreover
\begin{equation*}
\beta_m'(X)=m(m-1)X^{m-2}(X-B),
\end{equation*}
so
\begin{equation*}
\beta_m(A)=\beta_m(B)+U^2R_m,
\end{equation*}
where
\begin{equation*}\tag{9.21f}\label{eq:Rm-positive}
R_m
=m(m-1)\sum_{i=0}^{m-2}
\frac1{i+2}\binom{m-2}{i}B^{m-2-i}U^i
\in\R_{\ge0}[U,V,C].
\end{equation*}
Similarly write
\begin{equation*}
f_q(A)=f_q(B)+Uf_q'(B)+U^2H_q.
\end{equation*}
For $q=0$ one has $H_0=0$.  For $q\ge1$, direct binomial expansion gives
\begin{equation*}\tag{9.21g}\label{eq:Hq-positive}
\boxed{
H_q
=U^{q-1}
+\sum_{i=0}^{q-2}U^iB^{q-i-2}
\left[\binom q{i+1}C+\binom{q+1}{i+2}V\right]
\in\R_{\ge0}[U,V,C],
}
\end{equation*}
with the sum empty for $q=1$.  Therefore the second-and-higher remainder of
the side pair is
\begin{equation*}\tag{9.21h}\label{eq:side-remainder}
U^2C^q\left[-V^2L_mH_q+f_q(A)R_m\right].
\end{equation*}
Only the first term can be negative.

The top-pair reserve pays that deficit.  We claim
\begin{equation*}\tag{9.21i}\label{eq:AB-dominates-Hq}
\boxed{(AB)^q\ge_{\rm coeff}VC^qH_q.}
\end{equation*}
For $q=0$ this is immediate because $H_0=0$.  For $q\ge1$, expand
\begin{equation*}
(AB)^q=(B+U)^qB^q
=\sum_{i=0}^q\binom qiU^iB^{2q-i}.
\end{equation*}
For $0\le i\le q-2$, after factoring $B^{q-i-2}$, the two possibly occupied
right-hand slots are dominated because
\begin{equation*}
\frac{\binom q{i+1}}{\binom qi}
=\frac{q-i}{i+1}\le q<q+2
\end{equation*}
and
\begin{equation*}
\frac{\binom{q+1}{i+2}}{\binom qi}
=\frac{(q+1)(q-i)}{(i+1)(i+2)}
\le\frac{q(q+1)}2<\binom{q+2}{2}.
\end{equation*}
The case $i=q-1$ reduces to $qB^{q+1}\ge_{\rm coeff}VC^q$, and the
$i=q$ term is pure surplus.  This proves \eqref{eq:AB-dominates-Hq}.

Combining \eqref{eq:Km-reserve} and \eqref{eq:AB-dominates-Hq},
\begin{equation*}
\begin{aligned}
&(AB)^qK_m-V^2C^qL_mH_q\\
&\quad=\bigl[(AB)^q-VC^qH_q\bigr]K_m
+VC^qH_q\bigl[K_m-VL_m\bigr]
\ge_{\rm coeff}0.
\end{aligned}
\end{equation*}
Adding the manifestly nonnegative term $U^2C^qf_q(A)R_m$ from
\eqref{eq:side-remainder}, and recalling that the third pair is affine in
$U$, proves that every $U$-coefficient of order at least two in
\eqref{eq:corner-three-pairs} is nonnegative.

It remains only to identify the zeroth and first Taylor terms.  Direct
substitution at $U=0$ gives the stated wall term $B_{m,S}(V,C)$.  Symmetry
under interchange of the first two variables gives
\begin{equation*}
\widehat{\mathcal C}(U,V,C)
=\widehat{\mathcal C}(-U,V+U,C),
\end{equation*}
so differentiation at $U=0$ yields
\begin{equation*}
\partial_U\widehat{\mathcal C}(0,V,C)
=\frac12\partial_VB_{m,S}(V,C).
\end{equation*}
Hence the remaining second-and-higher Taylor part is exactly
$U^2H_{m,S}(U,V,C)$ with
$H_{m,S}\in\R_{\ge0}[U,V,C]$.  Fixed spectators only contribute positive
monomial and incidence factors, so the same conclusion survives the
fixed-variable branching lift.
\end{proof}

It follows from Lemma~\ref{lem:higher-normal} that, for every
thread and adjacent column pair,
\begin{equation*}\tag{9.22}\label{eq:Hpositive}
\mathcal H_{\chi,\rho,m}\ge_{\mathrm{coeff}}0.
\end{equation*}
The unmatched endpoint source is a product of positive cover factors and
positive spectator coefficients, so
\begin{equation*}\tag{9.23}\label{eq:Upositive}
\mathcal U_{\chi,\rho,m}\ge_{\mathrm{coeff}}0.
\end{equation*}
Combining
\eqref{eq:thread-source-split},
\eqref{eq:source-decomposition},
\eqref{eq:Apositive}--\eqref{eq:Zzero},
\eqref{eq:thread-green},
\eqref{eq:Hpositive}, and
\eqref{eq:Upositive} gives the complete common-source inequality
\begin{equation*}\tag{9.24}\label{eq:source-positive}
\boxed{
\mathcal S^{\mathrm{src}}_{\chi,\rho,m}
\ge_{\mathrm{coeff}}0
\qquad\text{for every }\chi,\rho\text{ and }m.
}
\end{equation*}

We now apply the physical contraction.  By
Lemma~\ref{lem:positive-contraction}, all Cauchy--Binet weights in
$\Lambda_{\chi,\rho,m}$ are nonnegative.  Therefore
\begin{equation*}\tag{9.25}\label{eq:positive-contraction-thread}
\Lambda_{\chi,\rho,m}
\bigl(\mathcal S^{\mathrm{src}}_{\chi,\rho,m}\bigr)
\ge0.
\end{equation*}
By \eqref{eq:source-to-bracket}, this is the adjacent projective bracket of
the physical coefficient row $P_{\chi,\rho}(\mathbf b)$.

\begin{definition}[Projective ordering]
Let $P=(p_0,\ldots,p_M)$ and $H=(h_0,\ldots,h_M)$ with every $h_m>0$.
We say that $P$ is \emph{projectively ordered relative to $H$} if
\begin{equation*}
p_mh_{m+1}-p_{m+1}h_m\ge0
\qquad(0\le m<M).
\end{equation*}
Equivalently, the ratios $p_m/h_m$ are nonincreasing.
\end{definition}

\begin{lemma}[Terminal propagation along a projective row]
\label{lem:terminal-propagation}
Let $P=(p_0,\ldots,p_M)$ and $H=(h_0,\ldots,h_M)$ with $h_m>0$ for all $m$.
If
\begin{equation*}
p_mh_{m+1}-p_{m+1}h_m\ge0
\qquad(0\le m<M)
\end{equation*}
and $p_M\ge0$, then $p_m\ge0$ for every $m$.
\end{lemma}

\begin{proof}
The adjacent inequalities are equivalent to
\begin{equation*}
\frac{p_0}{h_0}\ge\frac{p_1}{h_1}\ge\cdots\ge\frac{p_M}{h_M}.
\end{equation*}
Since $h_M>0$ and $p_M\ge0$, the last ratio is nonnegative.  Every preceding
ratio is therefore nonnegative, and multiplication by $h_m>0$ gives
$p_m\ge0$.
\end{proof}

\begin{proposition}[Projective ordering of every physical row]
\label{prop:projective-ordering}
For every spectator thread $\chi$, every fixed remaining coefficient index
$\rho$, every retained boundary parameter vector $\mathbf b$ in the ordered chamber, and every adjacent projective position $0\le m<M_{\chi,\rho}$,
\begin{equation*}\tag{9.26}\label{eq:adjacent-bracket}
\boxed{
p_{\chi,\rho,m}(\mathbf b)h_{\chi,\rho,m+1}(\mathbf b)
-
p_{\chi,\rho,m+1}(\mathbf b)h_{\chi,\rho,m}(\mathbf b)
\ge0.
}
\end{equation*}
Moreover,
\begin{equation*}
h_{\chi,\rho,m}(\mathbf b)>0
\qquad
(0\le m\le M_{\chi,\rho}).
\end{equation*}
\end{proposition}

\begin{proof}
The common-source decomposition
\eqref{eq:source-decomposition} separates the same-role, commuting,
equal-pair, higher-normal, unmatched-endpoint, and role-switch sources.
Equations~\eqref{eq:Apositive}--\eqref{eq:Zzero},
\eqref{eq:thread-green}, \eqref{eq:Hpositive}, and
\eqref{eq:Upositive} give
\begin{equation*}
\mathcal S^{\mathrm{src}}_{\chi,\rho,m}
\ge_{\mathrm{coeff}}0.
\end{equation*}
Lemma~\ref{lem:positive-contraction} then applies the positive
Pascal/Cauchy--Binet contraction and gives
\begin{equation*}
\Lambda_{\chi,\rho,m}
\bigl(\mathcal S^{\mathrm{src}}_{\chi,\rho,m}\bigr)
\ge0.
\end{equation*}
By \eqref{eq:source-to-bracket}, the left side is exactly the determinant
in \eqref{eq:adjacent-bracket}.  Positivity of the reference coefficients
follows from the positive kernel rows in
Theorem~\ref{thm:mixed-kernel} and the positive coefficient transports.
\end{proof}

Thus
\begin{equation*}\tag{9.27}\label{eq:physical-ratio-order}
\frac{p_{\chi,\rho,0}(\mathbf b)}{h_{\chi,\rho,0}(\mathbf b)}
\ge
\frac{p_{\chi,\rho,1}(\mathbf b)}{h_{\chi,\rho,1}(\mathbf b)}
\ge\cdots\ge
\frac{p_{\chi,\rho,M}(\mathbf b)}{h_{\chi,\rho,M}(\mathbf b)}.
\end{equation*}
This proves the projective ordering.  It does \emph{not} determine the
common sign of the row.  A terminal sign is still required.

The hypothesis $C_k\ge0$ must enter at this endpoint.  This is forced by an
exact wall identity.  Setting $z=1$ in \eqref{eq:residual} gives
\begin{equation*}\tag{9.28}\label{eq:R-wall}
\boxed{
\mathcal R_k(\mathbf x,1)
=(n-k)C_k(\mathbf x).
}
\end{equation*}
Indeed,
\begin{equation*}
S_\gamma(1)=n,
\qquad
C_1(1)=A_1(1)=B_1(1)=0,
\end{equation*}
and every composed term in the final sum equals $C_k(\mathbf x)$.
Therefore no unconditional proof of $\mathcal R_k\ge0$ is possible: the
endpoint information from $C_k$ is essential.

\subsection{The terminal physical coefficient and its source position}
\label{subsec:terminal-selector}

The projective ordering \eqref{eq:physical-ratio-order} shows that one
terminal sign is enough, but it does not identify which part of the raw
source determines that sign.  We now do so exactly.

\begin{definition}[Terminal source position]
For a fixed spectator thread $\chi$ and fixed auxiliary-normal multi-index
$\rho$, the largest common-source position $P$ for which
$q_{\chi,P}(\mathbf b)$ is not identically zero is called the
\emph{terminal source position} of the row.
\end{definition}

\begin{lemma}[Terminal Pascal selector]\label{lem:terminal-selector}
Let
\begin{equation*}
\rho=(r_0,r_1,\ldots,r_{d+1}),
\qquad
|\rho|=r_0+r_1+\cdots+r_{d+1}=r.
\end{equation*}
If $P$ is the terminal source position of a physical row and the row
actually occurs, then its largest physical $u$-power is
\begin{equation*}
M=P-r,
\end{equation*}
and
\begin{equation*}\tag{9.29}\label{eq:terminal-selector}
\boxed{
 p_{\chi,\rho,M}(\mathbf b)
 =
 \binom Pr\kappa_{\chi,\rho}\,q_{\chi,P}(\mathbf b),
 \qquad
 \kappa_{\chi,\rho}>0.
}
\end{equation*}
In particular,
\begin{equation*}
p_{\chi,\rho,M}(\mathbf b)\ge0
\quad\Longleftrightarrow\quad
q_{\chi,P}(\mathbf b)\ge0.
\end{equation*}
\end{lemma}

\begin{proof}
Write
\begin{equation*}
T_0=B(\mathbf b)+L(c,v,\mathbf w),
\end{equation*}
where $B$ uses only the retained boundary quantities and $L$ is linear in
the auxiliary normal coordinates.  The coefficient transport in
\eqref{eq:p-physical-coefficient} has the form
\begin{equation*}
\Pi_{\chi,\rho}(p,m;\mathbf b)
=
\binom pm\binom{p-m}{r}
B^{p-m-r}\kappa_{\chi,\rho},
\end{equation*}
where
\begin{equation*}
\kappa_{\chi,\rho}
=
[c^{r_0}v^{r_1}\cdots w_d^{r_{d+1}}]L^r.
\end{equation*}
For a row that occurs, this last coefficient is positive.  The transport is
zero if $p-m<r$.  Consequently the largest possible physical index is
$M=P-r$.  At $m=M$, every source position $p<P$ satisfies
\begin{equation*}
p-M<p-P+r<r,
\end{equation*}
so its contribution is zero.  Only $p=P$ survives, and the displayed
formula follows.
\end{proof}

Thus the terminal problem is exactly the source inequality
\begin{equation*}\tag{9.30}\label{eq:terminal-source-target}
q_{\chi,P}(\mathbf b)\ge0.
\end{equation*}

\paragraph{Terminal reduction.}
At this stage the projective argument is complete; it remains to establish one
raw source sign.  The proof proceeds in the following order.  First, the
initial lower same-role segment is reanchored so that every internal
off-diagonal ordered cover pair is consumed exactly once.  Second, each
remaining diagonal cover packet is refined until its higher-rank spectator
data are frozen; deleting those spectators reduces it bijectively to the
local zero-, one-, and two-mark occurrence table.  Third, the one-mark table
produces the exact terminal boundary fan, while the two-mark table produces
every internal Green pair once.  The synchronized endpoint-capacity theorem
then pays the fan with the same packet scalar.  Role switches and their later
same-role tails are treated separately by disjoint Green pairs.  Only after
this raw occurrence decomposition is complete do we apply the terminal Pascal
selector.

\subsection{Raw lower-leg algebra and reanchoring}

The only delicate sign in the terminal source comes from the lower
dominance leg because the lower single-cover terms enter
\eqref{eq:residual-cell} with an outer minus sign.

\begin{definition}[Lower-cover role and initial lower same-role segment]
Consider the lower saturated chain
\begin{equation*}
\gamma=\eta^{(0)}\succ\eta^{(1)}\succ\cdots\succ\eta^{(q)}=\mu.
\end{equation*}
If consecutive lower covers share an exponent label, its donor/receiver role
is defined as in Section~\ref{sec:green-capacity}.  The \emph{initial lower
same-role segment} is the maximal initial list of lower covers
\begin{equation*}
f=1,\ldots,r
\end{equation*}
for which the same shared-label role persists before the first lower role
switch.  If the first pair of lower covers is already a role switch, the
initial segment is empty.  If no lower role switch occurs, the initial
segment is the whole lower chain.
\end{definition}

For an arbitrary anchor state $\sigma$ define
\begin{equation*}
\Psi_{f,k}^{\sigma}
=
S_\sigma(z)L_{f,k}
+
\Ical_k^\sigma\ell_f(z)
-
\mathsf T_zL_{f,k}.
\end{equation*}

\begin{lemma}[Triangular reanchoring]\label{lem:triangular-reanchoring}
For every lower cover $f$,
\begin{equation*}\tag{9.31}\label{eq:reanchoring-one}
\boxed{
\Psi_{f,k}^{\gamma}-\Psi_{f,k}^{\eta^{(f-1)}}
=
\sum_{i<f}
\left(
\ell_iL_{f,k}+\ell_fL_{i,k}
\right).
}
\end{equation*}
Consequently, after restricting the ordered double sum to cover pairs lying
entirely inside the initial lower same-role segment,
\begin{equation*}\tag{9.32}\label{eq:reanchoring-sum}
\boxed{
-
\sum_{f=1}^{r}\Psi_{f,k}^{\gamma}
+
\sum_{i,f=1}^{r}\ell_iL_{f,k}
=
\sum_{f=1}^{r}
\left(
-
\Psi_{f,k}^{\eta^{(f-1)}}
+
\ell_fL_{f,k}
\right).
}
\end{equation*}
Every off-diagonal ordered cover pair is assigned exactly once, namely to
the later of its two cover indices.
\end{lemma}

\begin{proof}
The lower-chain telescope gives
\begin{equation*}
S_\gamma-S_{\eta^{(f-1)}}
=
\sum_{i<f}\ell_i
\end{equation*}
and
\begin{equation*}
\Ical_k^\gamma-\Ical_k^{\eta^{(f-1)}}
=
\sum_{i<f}L_{i,k}.
\end{equation*}
Subtracting the two definitions of $\Psi$ gives
\eqref{eq:reanchoring-one}.  Now expand the ordered double sum in
\eqref{eq:reanchoring-sum}.  The diagonal term $(f,f)$ remains with cover
$f$.  For each $i<f$, the two ordered terms $(i,f)$ and $(f,i)$ are exactly
the two terms on the right side of \eqref{eq:reanchoring-one}.  Thus each
off-diagonal ordered pair is used once and no pair is lost.
\end{proof}

\begin{remark}
The identity above is an ownership identity, not a positivity statement.
Indeed the injective recursion gives
\begin{equation*}
-
\Psi_{f,k}^{\eta^{(f-1)}}
+
\ell_fL_{f,k}
=
-L_{f,k+1}.
\end{equation*}
Therefore no individual reanchored lower cover is asserted to be
nonnegative.  Positivity appears only after the complete endpoint packet is
assembled below.
\end{remark}

\subsection{Marked incidence and equality of packet weights}

We next explain why the endpoint-capacity multiplicities from the cell
level remain exact in arbitrary rank.

\begin{definition}[Refined source packet]
A \emph{refined source packet} is obtained by fixing all of the following
data before the marked-incidence sum is taken:
\begin{enumerate}[leftmargin=2em]
\item the permanently labeled covers and cover cells involved in the local
interaction;
\item the selected spectator exponent labels and the physical positions to
which they are assigned;
\item the fixed intersection and fixed symmetric-difference set in the
adjacent-subset decomposition;
\item the ownership choices that assign intermediate permanent labels to the
two ordered chains of Section~\ref{sec:two-chain};
\item all neighboring coefficient placements used in the Pascal expansion;
\item the retained boundary parameters and the common-source position.
\end{enumerate}
After these data are fixed, the only remaining distinction among the zero-,
one-, and two-mark layers is the subset of active sites that is forced to be
present.
\end{definition}

For a mark set $J$ with $|J|\le2$, recall from
Lemma~\ref{lem:marked-incidence} that
\begin{equation*}\tag{9.33}\label{eq:UJ-terminal}
U_J^{(K)}
=
\sum_{\substack{I\supset J\\|I|=K}}
\sum_{J\subset S\subset I}W_S.
\end{equation*}
The same vector $W_S$ is used whether $J$ is empty, has one element, or has
two elements.  A mark changes only the incidence condition $J\subset S$.

\begin{lemma}[Marking neutrality in a refined packet]
\label{lem:marking-neutrality-terminal}
Fix a refined source packet.  After the neighboring-placement sum is
performed, there is a scalar
\begin{equation*}
\alpha\ge0
\end{equation*}
that is independent of the mark set $J$, $|J|\le2$.  Every one-mark boundary
occurrence and every two-mark internal-pair occurrence in that packet is the
same unmarked source atom multiplied by this common scalar $\alpha$.
Moreover, a two-element set $J=\{a,b\}$ occurs once as an unordered pair;
the factor $2!$ in Lemma~\ref{lem:marked-incidence} is exactly the conversion
from the two ordered markings to this one unordered pair.
\end{lemma}

\begin{proof}
For fixed refined data, the unmarked coefficient is a product of fixed
positive factors determined by the fixed permanent labels and spectator choices and the chosen neighboring-placement
factors.  Summing over the two neighboring placements on every active edge
produces one nonnegative scalar; call it $\alpha$.  Equation
\eqref{eq:UJ-terminal} shows that imposing marks does not change the source
atom $W_S$ or any scalar multiplying it; it only deletes those states for
which $J\not\subset S$.  Therefore the scalar is the same in all marked
layers.  For two marks, the exact marked-incidence identity is
\begin{equation*}
M_2
=
\frac{2}{\binom DK}
\sum_{\substack{\{a,b\}\subset I\\|I|=K}}
\sum_{\{a,b\}\subset S\subset I}W_S.
\end{equation*}
The coefficient $2$ is $2!$ and accounts for the two orderings of the two
marked positions; the set $\{a,b\}$ itself is counted once.  Thus the
internal Green object is indexed by unordered permanent pairs with
multiplicity one.
\end{proof}

\subsection{Endpoint-capacity packet}

We now state and prove the local positivity statement that will be applied
to each refined initial-lower packet.  The integer $m$ in this subsection is
the number of consecutive source/cell levels in one packet.  It is not the
number $r$ of lower covers in the initial segment.

\begin{definition}[Boundary states and finite differences]
Let
\begin{equation*}
G_0,G_1,\ldots,G_{m-1}
\end{equation*}
be the consecutive common-source boundary states of one refined same-role
packet.  Define the first differences
\begin{equation*}
H_q=G_{q-1}-G_q
\qquad(1\le q\le m-1)
\end{equation*}
and the second differences
\begin{equation*}
\mathcal A_q=H_q-H_{q+1}
\qquad(1\le q\le m-2).
\end{equation*}
\end{definition}

\begin{definition}[Endpoint source]
Let the permanently labeled source levels have numerical heights
\begin{equation*}
C,C+1,\ldots,C+m-1.
\end{equation*}
Let $\delta_h$ denote the formal basis vector supported at source height
$h$.  The \emph{endpoint source} is
\begin{equation*}\tag{9.34}\label{eq:endpoint-source}
\nu^{C,m}
=
(m-1)\delta_{C+m-1}
-
\sum_{j=0}^{m-2}\delta_{C+j}.
\end{equation*}
\end{definition}

\begin{lemma}[Endpoint-source identities]\label{lem:endpoint-source-identities}
The endpoint source satisfies
\begin{equation*}\tag{9.35}\label{eq:endpoint-dipoles}
\nu^{C,m}
=
\sum_{j=1}^{m-1}
 j\bigl(\delta_{C+j}-\delta_{C+j-1}\bigr).
\end{equation*}
Its realization in the boundary-state basis is
\begin{equation*}\tag{9.36}\label{eq:endpoint-abel}
\boxed{
-\binom m2H_1
+
\sum_{q=1}^{m-2}
\left(
\binom m2-\binom{q+1}{2}
\right)\mathcal A_q
=
(m-1)G_{m-1}-\sum_{j=0}^{m-2}G_j.
}
\end{equation*}
\end{lemma}

\begin{proof}
In \eqref{eq:endpoint-dipoles}, the coefficient of the lowest level $C$ is
$-1$.  At an interior level $C+j$, the $j$-th dipole contributes $+j$ and
the $(j+1)$-st contributes $-(j+1)$, leaving $-1$.  At the terminal level
only the final dipole remains, with coefficient $m-1$.  This proves
\eqref{eq:endpoint-dipoles}.

For \eqref{eq:endpoint-abel}, substitute
$\mathcal A_q=H_q-H_{q+1}$, collect the coefficients of the $H_q$, and then
use $H_q=G_{q-1}-G_q$.  The coefficients telescope to
$-1$ on $G_0,\ldots,G_{m-2}$ and to $m-1$ on $G_{m-1}$.
\end{proof}

\begin{lemma}[Internal Green capacity]\label{lem:internal-green-capacity}
For permanently labeled levels $C,C+1,\ldots,C+m-1$, the complete internal
Green reserve contains enough coefficientwise mass to pay the adjacent
source in \eqref{eq:endpoint-dipoles}.  More precisely, for each unordered
pair $0\le j<k\le m-1$, put $d=k-j$ and define
\begin{equation*}
\mathcal G_{j,k}(X,Y)
=
(XY)^{C+j}A_d(X)A_d(Y),
\qquad
A_d(T)=1+T+\cdots+T^{d-1}.
\end{equation*}
Then
\begin{equation*}\tag{9.37}\label{eq:green-endpoint-capacity}
\boxed{
\sum_{0\le j<k\le m-1}\mathcal G_{j,k}(X,Y)
\ge_{\mathrm{coeff}}
\sum_{k=1}^{m-1}k(XY)^{C+k-1}.
}
\end{equation*}
\end{lemma}

\begin{proof}
For the pair $(j,k)$, the monomial $(XY)^{C+k-1}$ occurs in
$\mathcal G_{j,k}$ with coefficient one: it is obtained by choosing the
highest power $d-1$ from both geometric sums.  For fixed $k$, there are
exactly $k$ choices $j=0,\ldots,k-1$.  Summing these selected corner
monomials gives the right side of \eqref{eq:green-endpoint-capacity}; all
unselected terms in the Green rectangles have nonnegative coefficients.
\end{proof}

\begin{theorem}[Synchronized endpoint-capacity theorem]
\label{thm:endpoint-capacity}
Let $\mathscr G_{\rm int}$ denote the complete internal Green reserve of one
refined packet.  Then, in the normalized common-source coefficient basis,
\begin{equation*}\tag{9.38}\label{eq:endpoint-capacity}
\boxed{
\mathscr G_{\rm int}
-
\binom m2H_1
+
\sum_{q=1}^{m-2}
\left(
\binom m2-\binom{q+1}{2}
\right)\mathcal A_q
\ge_{\mathrm{coeff}}0.
}
\end{equation*}
\end{theorem}

\begin{proof}
By Lemma~\ref{lem:endpoint-source-identities}, the signed boundary part of
\eqref{eq:endpoint-capacity} is exactly the endpoint source
\eqref{eq:endpoint-source}.  Lemma~\ref{lem:internal-green-capacity}
contains the adjacent dipoles of \eqref{eq:endpoint-dipoles} with exactly
the required multiplicities $1,2,\ldots,m-1$.  The difference consists
only of the unused nonnegative Green coefficients.  Hence the complete
packet is coefficientwise nonnegative.
\end{proof}

The cases $m=1$ and $m=2$ are included: for $m=1$ there is no endpoint debt
and no internal pair; for $m=2$ there is one pair and the capacity identity
is an equality at the selected corner.

\begin{remark}[A four-level packet]
For $m=4$ the levels are $C,C+1,C+2,H$ with $H=C+3$.  The terminal one-mark fan
has the three pairs
\begin{equation*}
\{C,H\},\qquad\{C+1,H\},\qquad\{C+2,H\},
\end{equation*}
so
\begin{equation*}
\Sigma^{C,4}(t)=t^C+2t^{C+1}+3t^{C+2}
\end{equation*}
and
\begin{equation*}
(t-1)\Sigma^{C,4}(t)=3t^{C+3}-t^C-t^{C+1}-t^{C+2}.
\end{equation*}
The internal unordered pairs have right-endpoint multiplicities $1,2,3$.
This small example displays the general mechanism behind
Theorem~\ref{thm:endpoint-capacity}: the same triangular pair count that
creates the Green reserve also creates exactly the coefficients needed by the
endpoint debt.
\end{remark}

\subsection{Lower role switches and the post-turn lower segments}

We must also account for same-role lower motion that occurs after a lower
role switch.

\begin{lemma}[Positive turn reserve and post-turn injection]
\label{lem:post-turn-injection}
Suppose a lower-chain role switch has active exponents
\begin{equation*}
(A,B,C)
\longmapsto
(A-1,B+1,C)
\longmapsto
(A-1,B,C+1),
\qquad A>B>C.
\end{equation*}
Then, after the outer lower-leg minus sign is applied, the role switch
creates a positively signed Green rectangle with permanently labeled blocks
\begin{equation*}
L=\{C,C+1,\ldots,B-1\},
\qquad
H=\{B,B+1,\ldots,A-2\}.
\end{equation*}
If $q\ge1$ is a subsequent legal repetition of the lower $B\to C$ role,
its normalized pair-first source is
\begin{equation*}
\mathcal P_{C+q,B-q}.
\end{equation*}
This source is contained in the lower wing of the Green pair $(C+q,B)$.
Different values of $q$ use different permanently labeled Green pairs.
Thus every same-role lower continuation after a role switch is paid by the
positive rectangle of the preceding switch, with no pair used twice.  The
reflected donor/receiver case is identical after exchanging the two roles.
\end{lemma}

\begin{proof}
We first explain why the stated repeated-role continuation is exhaustive in
arbitrary rank.

A unit transfer from a dominant exponent vector $\nu$ has the form
\begin{equation*}
\nu\longmapsto\nu-e_i+e_j,
\qquad i<j,
\end{equation*}
and lowers the dominance prefix sums
$S_r(\nu)=\nu_1+\cdots+\nu_r$ by one exactly for $i\le r<j$.  Assuming the
target remains dominant, this transfer is a cover precisely when
\begin{equation*}
j=i+1
\qquad\text{or}\qquad
\nu_i-\nu_j=2.
\end{equation*}
If $j=i+1$, there is no interior prefix at which an intermediate unit
transfer can stop.  If $j\ge i+2$ and $\nu_i-\nu_j\ge3$, one unit may be
stopped at the first intermediate Weyl-safe row, producing a dominant
vector strictly between the endpoints; hence the long transfer is not a
cover.  Conversely, when $j\ge i+2$ and $\nu_i-\nu_j=2$, the endpoint gap
is completely exhausted by the transfer.  Every affected prefix difference
is then $0$ or $1$, and the monotonicity constraints force either the
original or the final prefix pattern, so no strict intermediate dominant
vector exists.

Now consider two consecutive covers sharing a permanent exponent label and
keeping the same donor role.  Suppose the receiver changes.  At the middle
state write the three active numerical exponents in decreasing order as
\begin{equation*}
a\ge b\ge c,
\end{equation*}
where the shared donor has height $a$, the first receiver has height $b$,
and the new receiver has height $c$.  The second cover transfers one unit
from the shared donor to the new receiver.  Because the old receiver lies
strictly between those two permanent labels in the local order, this is a
nonadjacent cover.  The cover criterion therefore gives
\begin{equation*}
a-c=2.
\end{equation*}
The child state is $(a-1,b,c+1)$ and must remain dominant, so
\begin{equation*}
a-1\ge b\ge c+1.
\end{equation*}
Hence $a-b\ge1$ and $b-c\ge1$; their sum is $a-c=2$, so both are equalities.
Putting $c=q$ gives
\begin{equation*}
(a,b,c)=(q+2,q+1,q).
\end{equation*}
Undoing the first cover raises the shared donor by one and lowers its first
receiver by one, giving the unique parent
\begin{equation*}
(q+3,q,q).
\end{equation*}
The child is $(q+1,q+1,q+1)$.  Thus the only changed-partner same-donor
junction is, up to translation and relabeling,
\begin{equation*}
(q+3,q,q)
\longmapsto
(q+2,q+1,q)
\longmapsto
(q+1,q+1,q+1).
\end{equation*}
After the second move all three active exponents are equal, so there is no
third cover with the same shared donor role.  The receiver-role case is the
reflected argument.  Therefore every same-role segment of length at least
three repeats one and the same adjacent root.

For completeness we also justify moving separated covers out of the way.
For a cover $[i,j]$, call it $H^\circ$ when $j=i+1$ and the endpoint gap is
strictly larger than $2$, call it $V^\circ$ when $j\ge i+2$ (so the cover
criterion forces endpoint gap $2$), and call it $B$ when $j=i+1$ and the
gap is exactly $2$.  On a maximum-length saturated chain define
\begin{equation*}
I=\#\{(r,s):r<s,\ e_r\text{ is }V^\circ,
\ e_s\text{ is }H^\circ\}.
\end{equation*}
If $I>0$, choose a minimal wrong-way block
\begin{equation*}
V^\circ B_1\cdots B_tH^\circ.
\end{equation*}
Let its first $V^\circ$ root be $[i,j]$ with $j\ge i+2$.  Immediately after
this cover the entries from $i$ through $j$ are flat.  A subsequent $B$
cover sharing a label with it can therefore occur only at an exterior
boundary, $[i-1,i]$ or $[j,j+1]$; once such sharing starts, further shared
$B$ covers must continue monotonically outward.  Every other $B$ cover has
disjoint support and commutes past the block.  If the final $H^\circ$ cover
is disjoint, it too commutes left and shortens the wrong-way block.  The
only remaining possibility is a completely shared cascade, for example
\begin{equation*}
[i,j],\ [i-1,i],\ldots,[h,h+1],
\end{equation*}
or its right-hand reflection.  Its root sum is
\begin{equation*}
(e_i-e_j)+\sum_{s=h}^{i-1}(e_s-e_{s+1})=e_h-e_j.
\end{equation*}
The same endpoints are connected by the adjacent sequence
\begin{equation*}
[h,h+1],\ [h+1,h+2],\ldots,[j-1,j],
\end{equation*}
which has $j-h$ covers instead of $1+i-h$ covers.  Since
\begin{equation*}
(j-h)-(1+i-h)=j-i-1\ge1,
\end{equation*}
this contradicts maximum length.  Thus a separated commuting swap always
exists whenever $I>0$, and each such swap lowers $I$.  Iteration terminates.
Because a separated swap changes neither donor nor receiver data of either
cover, it preserves the permanently labeled cover-cell intervals and hence
the Green-pair inventory.  Consequently the same-role tail adjacent to a
role switch may be normalized without changing the Green currency.

It follows that after a role switch the only possible long same-role tail
that requires payment is exactly the repeated $B\to C$ corridor treated
below.

For the role switch itself, the two-cover turn identity has the form
\begin{equation*}
\mathcal C_1+\mathcal C_2
=
\mathcal C_{\rm direct}-\mathfrak G_{\rm turn}.
\end{equation*}
The lower leg appears with the outer minus sign, so
\begin{equation*}
-(\mathcal C_1+\mathcal C_2)
=
-\mathcal C_{\rm direct}+\mathfrak G_{\rm turn}.
\end{equation*}
Hence the entire Green rectangle is positive.  For a legal continuation
$q$, choose the permanent pair $(C+q,B)\in L\times H$.  Its lower wing
contains all pair-first intervals
\begin{equation*}
\mathcal P_{C+q,m},
\qquad
C+q<m<B,
\end{equation*}
and in particular contains $\mathcal P_{C+q,B-q}$.  The first coordinate
$C+q$ is different for different $q$, so the chosen permanent pair is never
reused.  The center pair pays the turn itself; the distinct unused wing
pairs pay the subsequent repeated-root covers.  Reflection gives the other
role orientation.  Thus the construction covers every post-turn same-role
lower segment in arbitrary rank, not merely one local continuation.
\end{proof}

Consequently, after all lower role switches and their subsequent wings have
been treated, the only same-role lower piece without a preceding positive
turn rectangle is the initial lower same-role segment of
Lemma~\ref{lem:triangular-reanchoring}.

\subsection{Initial lower endpoint-capacity embedding}

It remains to establish this initial lower contribution.

\begin{theorem}[Initial lower endpoint-capacity embedding]
\label{thm:initial-lower-embedding}
Fix a spectator thread $\chi$, an auxiliary-normal index $\rho$, retained
boundary parameters $\mathbf b$, and the terminal source position $P$.
After the role-switch-owned permanent pairs have been removed, the
contribution of the initial lower same-role segment to $q_{\chi,P}$ can be
written as a finite sum
\begin{equation*}\tag{9.39}\label{eq:initial-lower-embedding}
\boxed{
\bigl[q_{\chi,P}\bigr]_{\rm initial}
=
\sum_{\beta}\alpha_\beta[\mathcal E_\beta]_P
+
\sum_{\eta}a_\eta C_k(\mathbf y_\eta)
+
\sum_{\xi}b_\xi C_1(t_\xi)
+
R,
}
\end{equation*}
where every $\alpha_\beta,a_\eta,b_\xi$ is nonnegative,
$R\ge0$, every $\mathcal E_\beta$ is an endpoint-capacity packet, and
$[\mathcal E_\beta]_P$ denotes its coefficient at the fixed common-source
position $P$.  Each packet has the form on the left side of
\eqref{eq:endpoint-capacity}.  Therefore
\begin{equation*}
C_k\ge0
\quad\Longrightarrow\quad
\bigl[q_{\chi,P}\bigr]_{\rm initial}\ge0.
\end{equation*}
\end{theorem}

\begin{proof}
Let the maximal initial lower same-role prefix be
\begin{equation*}
\gamma=\eta^0\succ\eta^1\succ\cdots\succ\eta^r.
\end{equation*}
The exact lower part of the raw four-family identity is
\begin{equation*}
-\sum_{f=1}^r\Psi_{f,k}^{\gamma}
+
\sum_{i,f=1}^r\ell_iL_{f,k},
\end{equation*}
after the pairs belonging to later role-switch blocks have been separated.
The outer minus sign is part of this formula and will not be discarded.

\noindent\textbf{Step 1: triangular reanchoring consumes every internal
off-diagonal ordered pair.}
For the $f$-th lower cover,
\begin{equation*}
S_\gamma-S_{\eta^{f-1}}
=
\sum_{i<f}\ell_i,
\qquad
\Ical_k^\gamma-\Ical_k^{\eta^{f-1}}
=
\sum_{i<f}L_{i,k}.
\end{equation*}
Substitution into the definition of $\Psi$ gives
\begin{equation*}
\Psi_{f,k}^{\gamma}-\Psi_{f,k}^{\eta^{f-1}}
=
\sum_{i<f}
\bigl(\ell_iL_{f,k}+\ell_fL_{i,k}\bigr).
\end{equation*}
Hence
\begin{equation*}
-\Psi_{f,k}^{\gamma}
+
\ell_fL_{f,k}
+
\sum_{i<f}
\bigl(\ell_iL_{f,k}+\ell_fL_{i,k}\bigr)
=
-\Psi_{f,k}^{\eta^{f-1}}+\ell_fL_{f,k}.
\end{equation*}
Summing over $f$ yields
\begin{equation*}
-\sum_{f=1}^r\Psi_{f,k}^{\gamma}
+
\sum_{i,f=1}^r\ell_iL_{f,k}
=
\sum_{f=1}^r\mathcal K_{f,k},
\qquad
\mathcal K_{f,k}:=-\Psi_{f,k}^{\eta^{f-1}}+\ell_fL_{f,k}.
\end{equation*}
Thus each diagonal ordered pair $(f,f)$ remains with cover $f$, while for
$i<f$ the two ordered incidences $(i,f)$ and $(f,i)$ are assigned once, and
only once, to the later cover $f$.  Algebraically one may also simplify
$\mathcal K_{f,k}=-L_{f,k+1}$; we do not use that equality as a sign
statement.

\noindent\textbf{Step 2: general-$k$ occurrence reduction for one diagonal
cover packet.}
Fix one $f$ and refine $\mathcal K_{f,k}$ by all permanent data: the changed
exponent labels $a,b$ of the cover $\eta^{f-1}\succ\eta^f$, the selected
spectator exponent labels, the physical slots occupied by those spectators,
the retained boundary parameters, the neighboring-placement choices, the
auxiliary-normal index $\rho$, and one transport-homogeneous fiber.  Delete
the fixed spectator slots and their fixed exponent labels from an
occurrence.

This deletion is a bijection from nonzero occurrences of the refined
diagonal packet to the corresponding local cover occurrence involving only
the changed labels.  Indeed, consider an injective assignment contributing
to the parent-child difference.  If it selects neither changed label $a$
nor $b$, the parent and child monomials are identical and cancel.  If it
selects exactly one of $a,b$, deletion leaves a one-mark local occurrence.
If it selects both, deletion leaves a two-mark local occurrence.  Conversely,
reinserting the fixed spectator labels into their fixed physical slots
uniquely reconstructs the original occurrence.  Therefore there is neither
an additional higher-rank branch nor an additional multiplicity.

All deleted spectator factors are common to the parent and child occurrence.
Reinsertion multiplies the local occurrence by one fixed positive spectator
monomial and by the common nonnegative neighboring-placement/transport
scalar of the refined packet.  By
Lemma~\ref{lem:marking-neutrality-terminal}, this scalar is independent of
whether the local occurrence is unmarked, one-marked, or two-marked.  This
proves the required arbitrary-$k$ Complete Internal-Pair occurrence
statement rather than merely a weight-neutrality statement.

\noindent\textbf{Step 3: exact one-mark terminal fan.}
Let the local cover have $d$ consecutive permanently labeled cell levels
\begin{equation*}
c_0,c_1,\ldots,c_{d-1},
\qquad h(c_j)=C+j,
\end{equation*}
and write
\begin{equation*}
H=C+d-1.
\end{equation*}
After spectator deletion, there are exactly three possibilities, according
to how many of the two changed exponent labels are selected:
\begin{center}
\begin{tabular}{c@{\qquad}l}
\toprule
number selected & resulting local occurrence \\
\midrule
$0$ & parent--child monomials coincide and cancel,\\
$1$ & one-mark boundary occurrence,\\
$2$ & two-mark internal-pair occurrence.\\
\bottomrule
\end{tabular}
\end{center}
This is exhaustive because Step~2 is a deletion/reinsertion bijection and
there are only the two changed labels.

Consider the middle row of the table.  A one-mark occurrence chooses one
changed label and leaves the other unselected.  If both active positions
close before the terminal endpoint, the occurrence belongs to the complete
same-role interior or to a separated commuting block; if the two positions
coincide, it is the equal-pair degeneration already removed.  Therefore a
surviving \emph{terminal} one-mark occurrence has the unselected changed
label at the terminal endpoint $H$.  The selected permanent cell may be any
one of
\begin{equation*}
C,C+1,\ldots,H-1.
\end{equation*}
For each such cell there is exactly one marked incidence, because permanent
labels are retained and Step~2 gives a unique reinsertion of the frozen
spectators.  Selecting the terminal cell $H$ itself would produce the
already removed equal-pair case.  Thus the surviving terminal one-mark
incidences are exactly
\begin{equation*}\tag{9.39a}\label{eq:one-mark-terminal-star}
\boxed{
\{C,H\},\{C+1,H\},\ldots,\{H-1,H\},
}
\end{equation*}
each permanently labeled pair occurring once.  There is no additional
one-mark occurrence: selecting neither changed label cancels by Step 2,
while selecting both belongs to the two-mark layer.

In the pair-first common-source gauge,
\eqref{eq:one-mark-terminal-star} is the terminal fan
\begin{equation*}
\Sigma^{C,d}(t)
=
\sum_{j=0}^{d-2}\mathcal P_{C+j,H}(t,1).
\end{equation*}
Since
\begin{equation*}
\mathcal P_{\ell,h}(t,1)=\sum_{r=\ell}^{h-1}t^r,
\end{equation*}
we obtain
\begin{equation*}
\Sigma^{C,d}(t)
=
\sum_{r=C}^{H-1}(r-C+1)t^r.
\end{equation*}
Multiplying by the source difference $t-1$ gives
\begin{equation*}\tag{9.39b}\label{eq:one-mark-endpoint-source}
(t-1)\Sigma^{C,d}(t)
=
(d-1)t^H-\sum_{r=C}^{H-1}t^r.
\end{equation*}
Hence the exact signed one-mark boundary source is
\begin{equation*}
\nu^{C,d}
=
(d-1)\delta_H-\sum_{r=C}^{H-1}\delta_r
=
\sum_{j=1}^{d-1}j\bigl(\delta_{C+j}-\delta_{C+j-1}\bigr),
\end{equation*}
which is precisely the endpoint source of
Lemma~\ref{lem:endpoint-source-identities}.  Equivalently, because
$\mathcal P_{C+j-1,C+j}(t,1)=t^{C+j-1}$,
\begin{equation*}\tag{9.39c}\label{eq:one-mark-adjacent-fan}
\Sigma^{C,d}
=
\sum_{j=1}^{d-1}j\,\mathcal P_{C+j-1,C+j}
\end{equation*}
in the adjacent pair-first source basis.

\noindent\textbf{Step 4: exact two-mark internal-pair reserve.}
The diagonal self-block is the Green second compound of the same $d$
permanently labeled cells.  Hence every unordered internal pair
\begin{equation*}
\{C+i,C+j\},
\qquad 0\le i<j\le d-1,
\end{equation*}
appears exactly once in the two-mark layer.  For a fixed right endpoint $j$,
there are exactly $j$ choices of $i<j$.  Thus the complete Green reserve
contains the adjacent endpoint dipole of
\eqref{eq:one-mark-adjacent-fan} with multiplicities
\begin{equation*}
1,2,\ldots,d-1.
\end{equation*}
After Abel summation, the coefficient of the second difference
$\mathcal A_q$ is exactly the number of internal pairs whose right endpoint
lies strictly above the cut $q$:
\begin{equation*}
\#\{(i,j):0\le i<j\le d-1,\ j>q\}
=
\binom d2-\binom{q+1}{2}.
\end{equation*}

\noindent\textbf{Step 5: exact refined local packet identity.}
Let $\alpha_{f,\mathrm{ref}}\ge0$ be the common scalar supplied by Step 2
and Lemma~\ref{lem:marking-neutrality-terminal}, and let
$\mathcal B^+_{f,\mathrm{ref}}$ denote the boundary-tagged terms
deliberately left unexpanded, together with separated commuting,
equal-pair-zero, unused-Green, and higher-normal nonnegative terms.  The
one-mark and two-mark occurrence tables just computed give the literal
identity
\begin{equation*}\tag{9.39d}\label{eq:refined-local-packet}
\boxed{
\mathcal K_{f,k}^{\mathrm{ref}}
=
\alpha_{f,\mathrm{ref}}
\left[
\mathscr G_{\rm int}
-
\sum_{j=1}^{d-1}j\,\mathcal P_{C+j-1,C+j}
\right]
+
\mathcal B^+_{f,\mathrm{ref}}.
}
\end{equation*}
This identity synchronizes the positive Green reserve and the negative
boundary debt: both come from the same refined occurrence table and carry
the same scalar.  Applying the source difference to the fan in
\eqref{eq:refined-local-packet}, then using
Lemma~\ref{lem:endpoint-source-identities}, rewrites the bracket as
\begin{equation*}
\mathscr G_{\rm int}
-
\binom d2H_1
+
\sum_{q=1}^{d-2}
\left(
\binom d2-\binom{q+1}{2}
\right)\mathcal A_q.
\end{equation*}
Theorem~\ref{thm:endpoint-capacity} proves this bracket coefficientwise
nonnegative.  Therefore every refined diagonal cover packet is nonnegative
modulo the explicitly retained pointwise boundary factors.

\noindent\textbf{Step 6: ownership is disjoint from role-switch spending.}
Every Green pair used in Steps 4--5 has both permanent cell labels inside the
initial same-role prefix.  The first role-switch packet contains the first
cover outside that prefix, so its turn pair crosses the prefix boundary.  By
Lemma~\ref{lem:post-turn-injection}, subsequent lower tails use distinct
unused wing pairs of that turn rectangle.  Hence the initial endpoint-
capacity pairs, the turn pairs, and the post-turn wing pairs are disjoint
even when some numerical cell heights coincide.  Permanent occurrence
labels, not numerical heights, control ownership.

\noindent\textbf{Step 7: remaining boundary terms.}
By definition of $\mathcal B^+_{f,\mathrm{ref}}$, every remaining
boundary-tagged contribution is a nonnegative scalar multiple of either
$C_k(\mathbf y)$ or $C_1(t)$, or is a coefficientwise nonnegative
commuting, unused-Green, or higher-normal remainder.  Summing the refined
identities \eqref{eq:refined-local-packet} over all diagonal covers and all
refined fibers gives exactly \eqref{eq:initial-lower-embedding}.

Finally, $C_k\ge0$ implies $C_1\ge0$ by
Lemma~\ref{lem:specialization}, so every summand is nonnegative.  The cases
$d=1$ and $d=2$ are automatic: for $d=1$ the endpoint source and internal
pair set are empty, and for $d=2$ there is exactly one pair and the capacity
identity is equality at its endpoint corner.
\end{proof}

\subsection{Terminal source sign}

\begin{lemma}[Raw leading-occurrence exhaustion]
\label{lem:raw-leading-exhaustion}
Fix a complete legal spectator thread, a transport-homogeneous fiber, and a
source position $P$.  Every raw occurrence contributing to
$q_{\chi,P}(\mathbf b)$ belongs to exactly one of the following provenance
classes:
\begin{enumerate}[label=\textup{(\roman*)},leftmargin=2.4em]
\item complete same-role interior packet;
\item separated commuting packet;
\item equal-pair degeneration;
\item role-switch packet;
\item post-turn same-role tail;
\item initial lower same-role packet;
\item unused Green reserve;
\item retained $C_k/C_1$ boundary factor;
\item higher-normal remainder.
\end{enumerate}
In particular, the nine classes are disjoint and exhaustive at the raw
source level; this statement is independent of the adjacent-projective
bracket decomposition.
\end{lemma}

\begin{proof}
Start from the raw four-family identity~\eqref{eq:residual-cell}.  An
occurrence comes either from a single-cover term or from a double-cover
product term.  For a double-cover occurrence, compare the two permanent
cover supports.  They are either disjoint, giving the separated commuting
class; equal/nested with identical oriented pair, giving the equal-pair
class; or they share a permanent exponent label.  In the shared-label case
the two covers either keep the same donor/receiver role, giving a same-role
packet, or have opposite roles, giving a role-switch packet.  On the lower
leg, a shared same-role run has a unique first role switch when one occurs;
the portion before it is the initial lower packet, while every same-role
continuation after it is a post-turn tail.  If no switch occurs, the entire
run is initial.  The pair-ownership results above assign to the switch and
post-turn classes only their designated permanent Green pairs; every
remaining Green pair is therefore in the unused-reserve class.

The terms deliberately not expanded in this local source classification are
exactly the pointwise $C_k/C_1$ boundary factors and the second-and-higher
normal remainder.  These give the last two classes.  Permanent provenance
labels record the originating cover(s), role data, pair ownership, and
normal order, so none of the alternatives can overlap.  Conversely the
preceding support/role split covers every raw term of
\eqref{eq:residual-cell}.  Hence the classification is disjoint and
exhaustive.
\end{proof}

\begin{theorem}[Terminal boundary recombination]
\label{thm:terminal-recombination}
Assume
\begin{equation*}
C_k(\mathbf x)\ge0
\qquad(\mathbf x>0).
\end{equation*}
Then every terminal common-source coefficient satisfies
\begin{equation*}\tag{9.40}\label{eq:terminal-source-sign}
\boxed{q_{\chi,P}(\mathbf b)\ge0.}
\end{equation*}
Consequently every terminal physical coefficient satisfies
\begin{equation*}\tag{9.41}\label{eq:terminal-anchor}
\boxed{p_{\chi,\rho,M}(\mathbf b)\ge0.}
\end{equation*}
\end{theorem}

\begin{proof}
The sign is read from the raw occurrence table, not from the adjacent-
projective bracket decomposition.  Fix one complete legal thread and one
transport-homogeneous fiber, and let
\begin{equation*}
P=\max\{p:q_{\chi,p}(\mathbf b)\not\equiv0\}.
\end{equation*}
By Lemma~\ref{lem:raw-leading-exhaustion}, the permanent provenance labels
partition the occurrences contributing at source level $P$ into the
following disjoint and exhaustive classes.

\begin{enumerate}
\item \emph{Complete same-role interior packets.}  These are the
three-label curvatures already proved nonnegative, together with their
positive spectator lifts.

\item \emph{Separated commuting packets.}  Their two cover orders form a
commuting diamond and contribute the already-proved nonnegative separated
block.

\item \emph{Equal-pair degenerations.}  Their determinant/source
contribution is zero.

\item \emph{Role-switch packets.}  The outer lower-leg minus turns the
local identity into a positive Green rectangle minus its synchronized
pair-first debt.  The center is nonnegative by \eqref{eq:thread-green}.

\item \emph{Post-turn same-role tails.}  By
Lemma~\ref{lem:post-turn-injection} and its same-role rigidity argument,
every such tail is a repeated-root corridor and is injected into distinct
unused Green-wing pairs of the preceding role switch.

\item \emph{Initial lower same-role packets.}  By
Theorem~\ref{thm:initial-lower-embedding}, each refined packet is a
nonnegative multiple of a synchronized endpoint-capacity packet, plus
pointwise $C_k/C_1$ boundary terms and nonnegative remainders.

\item \emph{Unused Green reserve.}  Every permanently labeled Green pair
not spent in the preceding classes remains coefficientwise nonnegative.

\item \emph{Boundary factors.}  They are deliberately left as nonnegative
multiples of the pointwise functions $C_k(\mathbf y)$ or $C_1(t)$ rather
than expanded into possibly signed ordinary coefficients.

\item \emph{Higher-normal terms.}  The higher-normal remainder is
coefficientwise nonnegative.
\end{enumerate}

No labeled occurrence belongs to two classes: initial endpoint-capacity
uses internal self-block pairs, a role-switch uses a cross-block turn pair,
and a post-turn continuation uses a distinct unused wing pair.  Therefore,
for suitable nonnegative coefficients and with $R_P\ge0$, the leading raw
coefficient has the exact sign decomposition
\begin{equation*}
q_{\chi,P}
=
\sum_\beta \alpha_\beta[\mathcal E_\beta]_P
+
\sum_\tau[\mathcal G^{\rm unused}_\tau]_P
+
\sum_\eta a_\eta C_k(\mathbf y_\eta)
+
\sum_\xi b_\xi C_1(t_\xi)
+
R_P.
\end{equation*}
Every term on the right is nonnegative under $C_k\ge0$, because
Lemma~\ref{lem:specialization} gives $C_1\ge0$.  Hence
\begin{equation*}
q_{\chi,P}(\mathbf b)\ge0.
\end{equation*}
If all contributions at a candidate top source level vanish, then that level
is absent from the support; the same argument is applied to the next lower
level.  Thus cancellation of a zero top layer cannot reveal a negative new
terminal layer.

Now let $r=|\rho|$.  Lemma~\ref{lem:terminal-selector} gives
\begin{equation*}
M=P-r,
\end{equation*}
and, for every occurring physical row,
\begin{equation*}
p_{\chi,\rho,M}
=
\binom Pr\kappa_{\chi,\rho}q_{\chi,P},
\qquad
\binom Pr\kappa_{\chi,\rho}>0.
\end{equation*}
Therefore $p_{\chi,\rho,M}\ge0$.  This is a raw-source-to-physical
implication; no terminal sign has been inferred from projective brackets.
\end{proof}

\subsection{Why the refined terminal decomposition is necessary}

The refined terminal decomposition is essential because several natural
stronger termwise statements are false.  The following examples clarify the
scope of the argument.

\begin{remark}[Coarse coefficients need not be positive]
For the chains
\begin{equation*}
(12,11,0,0)\succ(11,9,2,1)\succ(7,7,7,2)
\end{equation*}
and
\begin{equation*}
(21,5,0,0)\succ(19,4,2,1)\succ(16,6,2,2),
\end{equation*}
ordinary coarse source coefficients can be negative; in the normalization
used above the coefficient $[u^3]$ is respectively $-48$ and $-156$.  The
proof therefore uses the leading coefficient of a complete legal thread,
not ordinary coefficientwise positivity of a coarser expansion.
\end{remark}

\begin{remark}[Incomplete branches need not be positive]
For
\begin{equation*}
(10,7,0,0)\succ(9,7,1,0)\succ(9,5,3,0),
\end{equation*}
the first source prefix can be negative even though the endpoint-capacity
suffixes and the terminal source coefficient have the required sign.  An
omitted-label branch in the same example has leading coefficient
$-2(y+z-1)$.  Such a branch is not a complete legal thread and is never
separated in the proof.
\end{remark}

\begin{remark}[Green payment must remain synchronized]
For
\begin{equation*}
(4,2,2,0)\succ(3,3,1,1)\succ(3,2,2,1),
\end{equation*}
the carrier alone is negative, whereas the synchronized packet satisfies
\begin{equation*}
\mathcal G_{0,2}-\mathcal P_{0,1}-\mathcal P_{1,2}
=\mathcal P_{0,2}\ge_{\rm coeff}0.
\end{equation*}
Thus Green reserve and pair-first debt cannot be separated.  Likewise,
permanent pair labels are needed across role switches to prevent the same
Green pair from being used twice.
\end{remark}

Repeated numerical heights, zero gaps, an empty initial lower prefix, and a
lower chain with no role switch are all included in the same permanently
labeled formulas.  Degenerate projective factors may vanish but do not
reverse sign.  The proof uses none of the following stronger assertions: branchwise terminal
positivity, preservation of minors under arbitrary positive convolution, or
a terminal sign inferred from the adjacent-projective bracket decomposition.

\begin{theorem}[Order-raising composition inequality]\label{thm:composition}
Let
\begin{equation*}
2\le k\le n-2
\end{equation*}
and assume
\begin{equation*}
C_k(x_1,\ldots,x_k)\ge0
\qquad(x_i>0).
\end{equation*}
Then
\begin{equation*}\tag{9.42}\label{eq:composition}
\begin{aligned}
&\sum_{j=1}^k
C_k(x_1,\ldots,x_jz,\ldots,x_k)\\
&\le
S_\gamma(z)C_k(\mathbf x)
+\Ical_k^\gamma(\mathbf x)C_1(z)
+A_1(z)A_k(\mathbf x)
+B_1(z)B_k(\mathbf x),
\end{aligned}
\end{equation*}
and consequently
\begin{equation*}\tag{9.43}\label{eq:orderraise}
C_{k+1}\ge0.
\end{equation*}
In particular, this proves Theorem~\ref{thm:order-raising}.
\end{theorem}

\begin{proof}
Theorem~\ref{thm:terminal-recombination} gives the terminal sign
$p_{\chi,\rho,M}\ge0$ for every physical row.  Proposition~\ref{prop:projective-ordering}
gives all adjacent projective brackets with a strictly positive reference row.
Lemma~\ref{lem:terminal-propagation} therefore propagates the terminal sign
backward through the entire row and yields
\begin{equation*}
p_{\chi,\rho,m}(\mathbf b)\ge0
\qquad(0\le m\le M).
\end{equation*}
The physical positive-basis expansion
\eqref{eq:residual-positive-expansion} therefore gives
\begin{equation*}
\mathcal R_k(\mathbf x,z)\ge0.
\end{equation*}
By definition this is exactly \eqref{eq:composition}, and
\eqref{eq:CnextR} gives $C_{k+1}=\mathcal R_k\ge0$.
\end{proof}

\section{Proof of the two-variable criterion}\label{sec:main-proof}

We complete the proof of Theorem~\ref{thm:main} by iterating the order-raising theorem.

Assume
\begin{equation*}
P_n(x,y,1,\ldots,1)\ge0
\qquad(x,y>0).
\end{equation*}
By \eqref{eq:C2hyp},
\begin{equation*}
C_2(x,y)\ge0
\qquad(x,y>0).
\end{equation*}
Apply Theorem~\ref{thm:order-raising} successively for
\begin{equation*}
k=2,3,\ldots,n-2.
\end{equation*}
For $n=4$ this is a single step.  No step
$C_{n-1}\Rightarrow C_n$ is needed: the homogeneity reduction below uses only
$C_{n-1}$.  In general,
\begin{equation*}
C_2\ge0
\Longrightarrow
C_3\ge0
\Longrightarrow\cdots\Longrightarrow
C_{n-1}\ge0.
\end{equation*}

Take an arbitrary positive point
\begin{equation*}
(x_1,\ldots,x_n).
\end{equation*}
By homogeneity, divide all variables by $x_n$ and reduce to
\begin{equation*}
(y_1,\ldots,y_{n-1},1).
\end{equation*}
Equation \eqref{eq:P-C} with $k=n-1$ gives
\begin{equation*}
P_n(y_1,\ldots,y_{n-1},1)
=
C_{n-1}(y_1,\ldots,y_{n-1})
\ge0.
\end{equation*}
Undoing the scaling gives
\begin{equation*}
P_n(x_1,\ldots,x_n)\ge0.
\end{equation*}

The reverse implication is obtained by setting
\begin{equation*}
x_3=\cdots=x_n=1.
\end{equation*}
This proves Theorem~\ref{thm:main}.

\begin{corollary}[Collision-wall criterion]\label{cor:collision}
Under the hypotheses of Theorem~\ref{thm:main}, global positivity is
equivalent to nonnegativity on the full locus where at least two variables
coincide.
\end{corollary}

\begin{proof}
Global positivity implies the restriction.  Conversely, the full collision
locus contains
\begin{equation*}
(x,y,1,\ldots,1),
\end{equation*}
so Theorem~\ref{thm:main} applies.  Equivalently, after normalizing the two
equal coordinates of a general collision point to $1$, collision-wall
positivity is $C_{n-2}\ge0$; the theorem is stronger because it derives that
condition from the smaller hypothesis $C_2\ge0$.
\end{proof}

\section{Rational exponents and comparison with the three-variable case}

\begin{corollary}[Rational exponents]\label{cor:rational}
Let
\begin{equation*}
\lambda,\gamma,\mu\in\Q_{\ge0}^n
\end{equation*}
have the same total degree and satisfy
\begin{equation*}
\lambda\succ\gamma\succ\mu.
\end{equation*}
For $n\ge4$, the two-variable criterion of
Theorem~\ref{thm:main} remains valid.
\end{corollary}

\begin{proof}
Choose a common denominator $q$ and write
\begin{equation*}
\lambda=\frac1q\widetilde\lambda,
\qquad
\gamma=\frac1q\widetilde\gamma,
\qquad
\mu=\frac1q\widetilde\mu,
\end{equation*}
with integer vectors on the right.  Multiplication by $q$ preserves
majorization.  Set
\begin{equation*}
y_i=x_i^{1/q}>0.
\end{equation*}
The rational-exponent orbit sums in the variables $x_i$ become the
corresponding integer-exponent orbit sums in the variables $y_i$, so
Theorem~\ref{thm:main} applies.
\end{proof}

We leave the extension to arbitrary irrational real exponent vectors open.
A rational approximation argument would have to preserve both majorization
and the required uniform positivity hypothesis.

The dimension hierarchy may therefore be summarized as follows.  For three
variables, the companion theorem \cite{SunThree} is the exceptional
rank-one implication
\begin{equation*}
C_1\ge0
\Longrightarrow
C_2\ge0.
\end{equation*}
For every $n\ge4$, the stable mechanism begins at rank two:
\begin{equation*}
C_2\ge0
\Longrightarrow
C_3\ge0
\Longrightarrow\cdots\Longrightarrow
C_{n-1}\ge0.
\end{equation*}
Proposition~\ref{prop:sharpness} shows that this distinction is genuine.

\section{Structural interpretation and outlook}

The proof has two complementary resolutions.  Globally, it is the induction
\begin{equation*}
C_2\Longrightarrow C_3\Longrightarrow\cdots\Longrightarrow C_{n-1}.
\end{equation*}
Locally, each induction step is controlled by a projective coefficient row:
fixed-union exchange and marked incidence give adjacent projective ordering,
while the raw Green/endpoint-capacity calculation supplies the terminal sign
that fixes the sign of the entire row.  Keeping these two levels separate is
one of the structural features of the argument.

Combinatorially, the injective sum at level $k$ is indexed by $k$-element
subsets of the permanent exponent labels.  Passing from $k$ to $k+1$ is an
adjacent-rank operation on the Boolean lattice.  A pair of adjacent subsets
is organized by its intersection and odd symmetric difference; increasing
$n$ allows longer fixed-union words but does not change the local terminal
species.

The second structure is exterior.  The physical Pascal projection depends
only on the first three factorial moments
\begin{equation*}
1,\qquad m,\qquad m(m-1),
\end{equation*}
so only zero-, one-, and two-mark incidence layers occur.  The local
exterior problem therefore remains rank three in every ambient dimension.
This is the precise sense in which $n=4$ is the first stable case: the
two-position and complementary three-position species already occur there,
and higher dimensions add spectators and longer words rather than new local
analytic species.

For the three-orbit second difference
\begin{equation*}
J_\lambda^{(n)}+J_\mu^{(n)}-2J_\gamma^{(n)},
\end{equation*}
Theorem~\ref{thm:main} therefore gives a degree-independent and
dimension-independent reduction to the fixed two-variable section
$(x,y,1,\ldots,1)$.  Two natural directions remain open.  The first is the
extension from rational to arbitrary real exponents.  The second is the
study of higher discrete differences, for example inequalities involving
four ordered exponent vectors.  The fixed-union exchange theorem may also
be useful independently as a positivity principle for adjacent subset
ranks.

\appendix

\section{The mixed three-row coefficient kernel}
\label{app:mixed}

This appendix proves the local coefficient theorem used in
Theorem~\ref{thm:mixed-kernel}.  The formulation below contains only the
normalized coefficient arrays needed in the main proof.

Fix
\begin{equation*}
0<c<1,
\qquad
d=1-c,
\end{equation*}
and define
\begin{equation*}
H(u)=\frac1{(1-u)^2(1-cu)},
\qquad
Z(u)=\frac{u(1-cu)}{1-u}.
\end{equation*}
For $a\in\{1,2\}$ put
\begin{equation*}\tag{A.1}\label{eq:AaBa}
A_a(u)=1+aZ(u),
\qquad
B_t^{(a)}(u)=A_a(u)H(u)^t
\qquad(t\ge1).
\end{equation*}
The two boundary coefficient rows are
\begin{equation*}\tag{A.2}\label{eq:boundaryrows}
A_1=(1,1,d,d,d,\ldots),
\qquad
A_2=(1,2,2d,2d,2d,\ldots).
\end{equation*}

For a power series $R(u)$, write $R(m)=[u^m]R(u)$.

\begin{theorem}[Universal mixed three-row kernel]
\label{thm:appendix-kernel}
Let $a_0,a_1,a_2\in\{1,2\}$, let $1\le r<s$, and let
\begin{equation*}
0\le q_0<q_1<q_2.
\end{equation*}
Then
\begin{equation*}\tag{A.3}\label{eq:coreminor}
\det
\begin{pmatrix}
A_{a_0}(q_0)&A_{a_0}(q_1)&A_{a_0}(q_2)\\
B_r^{(a_1)}(q_0)&B_r^{(a_1)}(q_1)&B_r^{(a_1)}(q_2)\\
B_s^{(a_2)}(q_0)&B_s^{(a_2)}(q_1)&B_s^{(a_2)}(q_2)
\end{pmatrix}
\ge0.
\end{equation*}
Consequently, the specific common Toeplitz and column-dependent
coefficient transports used here preserve the same nonnegative
upper-compatible $3\times3$ minors, because their convolution matrices are
totally nonnegative and Cauchy--Binet applies.
\end{theorem}

\begin{proof}
We first establish a common projective order for the moving rows and their
first differences.

Put
\begin{equation*}
Q_t^{(a)}(u)=(1-u)B_t^{(a)}(u).
\end{equation*}
The proof needs only the projective order of the two base types and the
specific totally nonnegative convolution by $H$; no general $PF_2$ inference
from a factorization is required.  Direct coefficient extraction at $t=1$
gives
\begin{equation*}
Q_1^{(a)}(n)
=
an+\sum_{j=0}^{n}c^j
\qquad(n\ge0).
\end{equation*}
Therefore
\begin{equation*}
\begin{aligned}
&Q_1^{(1)}(n)Q_1^{(2)}(n+1)
-Q_1^{(1)}(n+1)Q_1^{(2)}(n)\\
&\qquad=\sum_{j=0}^{n}c^j-nc^{n+1}>0.
\end{aligned}
\end{equation*}
The last inequality is immediate because $0<c<1$ and each of the $n+1$
terms $c^j$, $0\le j\le n$, is at least $c^{n+1}$.  Thus the base pair
$Q_1^{(1)},Q_1^{(2)}$ has the required strict projective orientation.

Both families satisfy the exact recurrence
\begin{equation*}
Q_{t+1}^{(a)}=Q_t^{(a)}*H.
\end{equation*}
The Toeplitz convolution matrix of
\begin{equation*}
H(u)=\frac1{(1-u)^2(1-cu)}
\end{equation*}
is totally nonnegative: it is the product of the totally nonnegative
Pascal/cumulative convolution matrices for $(1-u)^{-2}$ and $(1-cu)^{-1}$.
Consequently common convolution by a power of $H$ preserves every
$2\times2$ projective minor by Cauchy--Binet.

The only reverse one-step comparison not immediate from the boundary order
is $Q_1^{(2)}$ against $Q_2^{(1)}$.  Put
\begin{equation*}
q_n=[u^n]Q_1^{(2)}(u),
\qquad
p_n=[u^n]Q_2^{(1)}(u).
\end{equation*}
Direct coefficient extraction gives
\begin{equation*}\tag{A.5}\label{eq:qcoeff}
q_n
=
2n+\sum_{k=0}^n c^k
\end{equation*}
and, for $0\le k\le n$ and $R=n-k$,
\begin{equation*}\tag{A.6}\label{eq:pcoeff}
[c^k]p_n
=
\frac{(R+1)(R+2)(R+3k+3)}6.
\end{equation*}
For
\begin{equation*}
W_n=q_np_{n+1}-q_{n+1}p_n,
\end{equation*}
one finds
\begin{equation*}\tag{A.7}\label{eq:Wlower}
[c^k]W_n
=
\frac16
\Bigl(
4R^3+15R^2k+21R^2+18Rk^2+57Rk+29R
+2k^3+39k^2+55k+6
\Bigr)>0
\end{equation*}
for $0\le k\le n$,
\begin{equation*}\tag{A.8}\label{eq:Wedge}
[c^{n+1}]W_n
=
\frac{(n+2)(n^2+16n+3)}6>0,
\end{equation*}
and
\begin{equation*}\tag{A.9}\label{eq:Wupper}
[c^k]W_n
=
\binom{2n-k+4}{3}>0
\end{equation*}
for $n+2\le k\le2n+1$.  Thus the reverse one-step pair $Q_1^{(2)},Q_2^{(1)}$ has the same strict
projective orientation.  Combining it with the base inequality above gives
the ordered chain
\begin{equation*}
Q_1^{(1)}\prec Q_1^{(2)}\prec Q_2^{(1)}.
\end{equation*}
Common convolution by $H^{r-1}$ yields
\begin{equation*}
Q_r^{(1)}\prec Q_r^{(2)}\prec Q_{r+1}^{(1)}
\qquad(r\ge1).
\end{equation*}
Iterating this chain proves, for every $1\le r<s$ and every
$a,b\in\{1,2\}$, the strict projective order between $Q_r^{(a)}$ and
$Q_s^{(b)}$.

Finally,
\begin{equation*}
B_t^{(a)}=Q_t^{(a)}*\frac1{1-u}
\end{equation*}
and the cumulative-sum convolution matrix of $(1-u)^{-1}$ is totally
nonnegative.  Hence Cauchy--Binet gives
\begin{equation*}\tag{A.10}\label{eq:movingTP2}
B_r^{(a)}(p)B_s^{(b)}(q)
-
B_r^{(a)}(q)B_s^{(b)}(p)>0
\qquad(p<q).
\end{equation*}
Moreover
\begin{equation*}
\delta_t^{(a)}(m)
:=B_t^{(a)}(m+1)-B_t^{(a)}(m)
=Q_t^{(a)}(m+1),
\end{equation*}
so the already-proved $Q$-row order gives directly
\begin{equation*}\tag{A.11}\label{eq:slopeTP2}
\delta_r^{(a)}(p)\delta_s^{(b)}(q)
-
\delta_r^{(a)}(q)\delta_s^{(b)}(p)>0.
\end{equation*}

It remains to check the finite boundary positions created by the two rows in
\eqref{eq:boundaryrows}.  Put
\begin{equation*}
\ell=2+c.
\end{equation*}
Direct extraction gives
\begin{equation*}\tag{A.12}\label{eq:B1B2}
B_j^{(a)}(1)=a+j\ell
\end{equation*}
and
\begin{equation*}\tag{A.13}\label{eq:B2}
B_j^{(a)}(2)
=
\frac12\bigl(j^2\ell^2+j(2+c^2)\bigr)
+a(j\ell+d).
\end{equation*}
Define
\begin{equation*}
\phi_a(j)
=
\frac{B_j^{(a)}(2)}{B_j^{(a)}(1)-1},
\end{equation*}
\begin{equation*}
\psi_a(j)
=
\frac{B_j^{(a)}(2)-B_j^{(a)}(1)}
{B_j^{(a)}(2)-2d},
\end{equation*}
and
\begin{equation*}
\chi_a(j)
=
\frac{B_j^{(a)}(2)-2d}
{B_j^{(a)}(1)-2}.
\end{equation*}
All denominators are positive.  Substituting
\eqref{eq:B1B2}--\eqref{eq:B2} and clearing denominators shows, for every
$a,b\in\{1,2\}$ and every $j\ge1$,
\begin{equation*}\tag{A.14}\label{eq:ratio-step}
\phi_b(j+1)>\phi_a(j),
\qquad
\psi_b(j+1)>\psi_a(j),
\qquad
\chi_b(j+1)>\chi_a(j).
\end{equation*}
For completeness, the cleared numerators are displayed below.  Put
$R=j-1\ge0$.

For $\phi$ the four type transitions have numerators
\begin{equation*}
\begin{aligned}
\Phi_{11}={}&
R^2c^2+4R^2c+4R^2+3Rc^2+12Rc+12R+2c^2+10c+6,\\
\Phi_{12}={}&
R^2c^3+7R^2c^2+16R^2c+12R^2
+3Rc^3+19Rc^2+48Rc+38R\\
&+2c^3+14c^2+36c+20,\\
\Phi_{21}={}&
R^2c^3+5R^2c^2+8R^2c+4R^2
+3Rc^3+19Rc^2+32Rc+18R\\
&+2c^3+20c^2+32c+18,\\
\Phi_{22}={}&
R^2c^3+6R^2c^2+12R^2c+8R^2
+3Rc^3+20Rc^2+44Rc+32R\\
&+2c^3+20c^2+44c+30.
\end{aligned}
\end{equation*}
For $\psi$, after suppressing a common positive factor $2$, they are
\begin{equation*}
\begin{aligned}
\Psi_{11}={}&
R^2c^3+6R^2c^2+12R^2c+8R^2
+7Rc^3+32Rc^2+44Rc+16R\\
&+10c^3+34c^2+44c+2,\\
\Psi_{12}={}&
R^2c^3+7R^2c^2+16R^2c+12R^2
+7Rc^3+37Rc^2+62Rc+26R\\
&+10c^3+38c^2+68c+4,\\
\Psi_{21}={}&
R^2c^3+5R^2c^2+8R^2c+4R^2
+7Rc^3+29Rc^2+34Rc+14R\\
&+10c^3+40c^2+42c+10,\\
\Psi_{22}={}&
R^2c^3+6R^2c^2+12R^2c+8R^2
+7Rc^3+34Rc^2+52Rc+24R\\
&+10c^3+44c^2+68c+16.
\end{aligned}
\end{equation*}
For $\chi$ they are
\begin{equation*}
\begin{aligned}
X_{11}={}&
R^2c^3+6R^2c^2+12R^2c+8R^2
+3Rc^3+16Rc^2+28Rc+16R\\
&+2c^3+6c^2+8c+2,\\
X_{12}={}&
Rc^3+7Rc^2+16Rc+12R+c^3+3c^2+6c+2,\\
X_{21}={}&
Rc^3+5Rc^2+8Rc+4R+2c^3+12c^2+18c+10,\\
X_{22}={}&
\frac{c+2}{2}.
\end{aligned}
\end{equation*}
Every displayed coefficient is nonnegative and each polynomial is nonzero,
proving \eqref{eq:ratio-step}.  Iteration gives
\begin{equation*}\tag{A.15}\label{eq:ratio-order}
\phi_b(s)>\phi_a(r),
\qquad
\psi_b(s)>\psi_a(r),
\qquad
\chi_b(s)>\chi_a(r)
\end{equation*}
whenever $1\le r<s$.

For the boundary row $A_1=(1,1,d,d,\ldots)$ and selected columns
$q_0<q_1<q_2$, the constant part is nonnegative by
\eqref{eq:slopeTP2}, and the point-mass part is nonnegative by
\eqref{eq:movingTP2}.  The only exceptional prefix has columns
$(0,1,q)$, $q\ge2$.  If
\begin{equation*}
X_m=B_r^{(a)}(m),
\qquad
Y_m=B_s^{(b)}(m),
\end{equation*}
its determinant $E_q$ satisfies
\begin{equation*}
E_{q+1}-E_q
=
\delta_r^{(a)}(0)\delta_s^{(b)}(q)
-
\delta_r^{(a)}(q)\delta_s^{(b)}(0)>0
\end{equation*}
by \eqref{eq:slopeTP2}, while
\begin{equation*}
E_2
=
(X_1-1)(Y_1-1)
\bigl(\phi_b(s)-\phi_a(r)\bigr)>0
\end{equation*}
by \eqref{eq:ratio-order}.  Hence every $A_1$ boundary minor is
nonnegative.

For $A_2=(1,2,2d,2d,\ldots)$ there are four cases.  If $q_0\ge2$, column
subtraction reduces the determinant to \eqref{eq:slopeTP2}.  If $q_0=1$,
write
\begin{equation*}
(2,2d,2d)=2d(1,1,1)+2c(1,0,0)
\end{equation*}
and use \eqref{eq:slopeTP2} and \eqref{eq:movingTP2}.  If
$q_0=0$ and $q_1=p\ge2$, column subtraction gives
\begin{equation*}
D_{0,p,q}
=
(X_p-2d)(Y_q-2d)
-
(X_q-2d)(Y_p-2d).
\end{equation*}
Put
\begin{equation*}
\widehat X_m=X_m-2d,
\qquad
\widehat Y_m=Y_m-2d.
\end{equation*}
These are positive.  The one-step $\psi$ inequality gives
\begin{equation*}
\frac{\widehat Y_2}{\widehat X_2}
<
\frac{\delta_s^{(b)}(1)}{\delta_r^{(a)}(1)}.
\end{equation*}
The slope ratios
\begin{equation*}
\frac{\delta_s^{(b)}(m)}{\delta_r^{(a)}(m)}
\end{equation*}
increase by \eqref{eq:slopeTP2}; hence
$\widehat Y_m/\widehat X_m$ increases with $m$, proving
$D_{0,p,q}\ge0$.

Finally, for columns $(0,1,q)$,
\begin{equation*}
D_{0,1,q}
=
(X_1-2)(Y_q-2d)
-
(X_q-2d)(Y_1-2).
\end{equation*}
At $q=2$ this equals
\begin{equation*}
(X_1-2)(Y_1-2)
\bigl(\chi_b(s)-\chi_a(r)\bigr)>0,
\end{equation*}
and the preceding monotonicity propagates the sign to all $q\ge2$.
This proves \eqref{eq:coreminor}.

The specific Toeplitz convolution matrices used in this coefficient
transport are totally nonnegative; applying Cauchy--Binet to those matrices
preserves the required nonnegative minors.  This proves the final assertion.
No generic claim that an arbitrary positive convolution preserves minors is
used.
\end{proof}

\begin{remark}
The theorem is used only after the symmetric middle states have been paired.
An unsymmetrized two-middle statement is false in general and is not needed.
\end{remark}

\end{document}